\documentclass[10pt]{amsart}
\usepackage[T1]{fontenc}
\usepackage[utf8]{inputenc}
\usepackage[english]{babel}
\usepackage{csquotes} 
\usepackage[a4paper,margin=3cm]{geometry}
\allowdisplaybreaks
\usepackage[draft]{graphicx}
\usepackage[backend=biber,style=alphabetic,bibencoding=utf8,language=auto,autolang=other,giveninits=true,doi=false,isbn=false,url=false,maxnames=10,sorting=nty]{biblatex}
\DeclareNolabel{
  \nolabel{\regexp{[\p{P}\p{S}\p{C}\p{Z}]+}}
}
\usepackage{caption} 
\usepackage[hyperfootnotes=false]{hyperref} 
\hypersetup{
	colorlinks = true,
	linkcolor = {blue},
	urlcolor = {red},
	citecolor = {blue}
}
\usepackage[nameinlink,capitalise,sort]{cleveref}
\crefname{equation}{}{} 

\crefname{figure}{Figure}{Figures}

\usepackage{amsmath}
\usepackage{amsthm}
\usepackage{amssymb}
\usepackage{mathrsfs} 
\usepackage{faktor} 
\usepackage{dsfont} 
\usepackage[dvipsnames]{xcolor}
\usepackage{array}
\usepackage{hhline}
\usepackage[normalem]{ulem}
\usepackage{comment} 
\usepackage[toc,page]{appendix} 
\usepackage{imakeidx} 
\usepackage{xparse} 
\usepackage{mathtools}
\usepackage{fancyhdr} 
\usepackage{ifthen} 
\usepackage{forloop} 
\usepackage{xstring}
\usepackage{tikz}
\usetikzlibrary{positioning, arrows.meta}
\usetikzlibrary{babel} 
\usepackage{tikz-cd} 
\usepackage{emptypage} 
\usepackage{listings} 
\usepackage{subcaption}
\usepackage{cancel}

\theoremstyle{plain}
\newtheorem{lemma}{Lemma}[section]
\newtheorem{proposition}[lemma]{Proposition}
\newtheorem{theorem}[lemma]{Theorem}

\newtheorem{remark}[lemma]{Remark}

\theoremstyle{definition}

\theoremstyle{remark}

\numberwithin{equation}{section}

\newcommand{\loc}{\mathrm{loc}}

\newcommand{\R}{\mathbb{R}}

\newcommand{\p}{\partial}
\newcommand{\pt}{\partial_t}

\newcommand{\bx}{\boldsymbol{x}}
\newcommand{\bb}{{\boldsymbol{b}}}

\newcommand{\be}{{\boldsymbol{e}}}

\newcommand{\bu}{{\boldsymbol{u}}}
\newcommand{\bv}{{\boldsymbol{v}}}
\newcommand{\bU}{{\boldsymbol{U}}}
\newcommand{\bw}{{\boldsymbol{w}}}
\newcommand{\bB}{{\boldsymbol{B}}}
\newcommand{\Bf}{{\boldsymbol{f}}}
\newcommand{\BF}{{\boldsymbol{F}}}

\newcommand{\bA}{\boldsymbol{A}}
\newcommand{\bg}{\boldsymbol{g}}
\newcommand{\bD}{\boldsymbol{D}}

\newcommand{\vp}{\varphi}

\newcommand{\ve}{\varepsilon}
\usepackage{soul}
\usepackage{xcolor}

\makeatletter
\soulregister\ref7
\soulregister\eqref7
\soulregister\cref7
\soulregister\Cref7
\soulregister\autoref7

\soulregister\today7
\soulregister\cite7
\soulregister\citep7
\soulregister\citet7
\soulregister\citealp7
\soulregister\citeauthor7
\soulregister\citeyear7

\soulregister\emph7
\soulregister\textbf7
\soulregister\textit7
\soulregister\texttt7
\makeatother

\makeatletter
\newcommand{\hlg@math}[1]{%
  \begingroup
  \setlength{\fboxsep}{0.6pt}
  \mathchoice
    {\colorbox{yellow!25}{$\displaystyle #1$}}%
    {\colorbox{yellow!25}{$\textstyle #1$}}%
    {\colorbox{yellow!25}{$\scriptstyle #1$}}%
    {\colorbox{yellow!25}{$\scriptscriptstyle #1$}}%
  \endgroup
}
\makeatother

\DeclareRobustCommand{\hlg}[1]{%
  \begingroup
  \sethlcolor{yellow!25}%
  \ifmmode
    \hlg@math{#1}%
  \else
    \hl{#1}%
  \fi
  \endgroup
}

\usepackage[shortlabels]{enumitem}
\newlist{thmenum}{enumerate}{1} 
\setlist[thmenum]{label=\textup{(\roman*)},
                  ref=\thetheorem-{(\roman*)}}
\crefname{thmenumi}{Theorem}{Theorem}

\newlist{lemenum}{enumerate}{1} 
\setlist[lemenum]{label={(\roman*)},
                  ref=\thelemma-{(\roman*)}}
\crefname{lemenumi}{Lemma}{Lemmas}

\newlist{propenum}{enumerate}{1} 
\setlist[propenum]{label={(\roman*)},
                  ref=\thelemma-{(\roman*)}}
\crefname{propenumi}{Proposition}{Propositions}

\begin{document}


\title[Landau solutions for MHD]{Asymptotic stability of Landau solutions to the MHD system and energy decay}

\author[N.~De Nitti]{Nicola De Nitti}
\address[N.~De Nitti]{Politecnico di Bari, Dipartimento di Meccanica, Matematica e Management, Via Edoardo Orabona 4, 70125 Bari (BA), Italy.}
\email{nicola.denitti@poliba.it}

\author[Y.~Wang]{Yun Wang}
\address[Y.~Wang]{Soochow University, School of Mathematical Sciences, Center for dynamical systems and differential equations, 215006 Suzhou, P.\,R.~China.}
\email[]{ywang3@suda.edu.cn}

\author[S.~Zhang]{Shaoheng Zhang}
\address[S.~Zhang]{Soochow University, School of Mathematical Sciences,
215006 Suzhou, P.\,R.~China.}
\email{20234007008@stu.suda.edu.cn}



\begin{abstract}
We consider the three-dimensional incompressible MHD system. Any weak solution satisfying a strong energy inequality is $L^2$-asymptotically stable around a Landau solution.
Under an additional integrability assumption on the initial perturbation, we also obtain an explicit algebraic decay rate for the $L^2$-norm of the velocity and magnetic perturbations.

\end{abstract}

\maketitle


\section{Introduction}
\label{sec:intro}

\subsection{The MHD and Navier--Stokes systems}\label{ssec:mhd}

The \textit{magnetohydrodynamic} (MHD) \textit{equations} describe the behavior of an electrically conducting, incompressible, viscous fluid and play a fundamental role in astrophysics, geophysics, plasma physics, and in certain industrial applications. An introduction to the topic can be found in \cite{Davidson_Book_2nd_2017}.

We consider the initial value problem for the MHD system in three space dimensions: 
\begin{align}\label{eq:IVPMHD}
\begin{cases}
\pt\bu + \bu \cdot \nabla \bu + \nabla p = \nu\, \Delta \bu + \bB \cdot \nabla \bB+\Bf, &  (\bx,t) \in \mathbb{R}^3\times(0,\infty), \\
\pt\bB + \bu \cdot \nabla \bB =  \bB \cdot \nabla \bu + \mu\, \Delta \bB, &  (\bx,t) \in \mathbb{R}^3\times(0,\infty), \\
\nabla \cdot \bu = 0, &  (\bx,t) \in \mathbb{R}^3\times(0,\infty), \\
\nabla \cdot \bB = 0, &  (\bx,t) \in \mathbb{R}^3\times(0,\infty), \\
\bu(\bx,0)=\bu_0(\bx), &  \bx \in \mathbb{R}^3,\\
\bB(\bx,0)=\bB_0(\bx), &  \bx \in \mathbb{R}^3.
\end{cases}
\end{align}
Here, the unknown functions $\bu,\bB,p$ denote the velocity, magnetic field, and pressure, respectively; the data $\bu_0,\bB_0$ and $\Bf$ represent the given initial velocity, initial magnetic field, and external force. 

The equations \cref{eq:IVPMHD}$_{1,3}$ constitute the Navier--Stokes system with an external Lorentz force, while \cref{eq:IVPMHD}$_{2,4}$ are derived from Maxwell's equations for the magnetic field $\boldsymbol{B}$. In particular, in the absence of the magnetic field (i.\,e., when $\boldsymbol{B} \equiv \boldsymbol{0}$), the MHD system \cref{eq:IVPMHD} reduces to the \emph{Navier--Stokes equations}, that is 
\begin{align}\label{eq:IVPNSE}
\begin{cases}
\pt\bu + \bu \cdot \nabla \bu + \nabla p =\nu \Delta \bu +\Bf, &  (\bx,t) \in \mathbb{R}^3\times(0,\infty), \\
\nabla \cdot \bu = 0, &  (\bx,t) \in \mathbb{R}^3\times(0,\infty), \\
\bu(\bx,0)=\bu_0(\bx), &  \bx \in \mathbb{R}^3.
\end{cases}
\end{align}

The existence of weak solutions to the Navier--Stokes system with zero external force ($\Bf\equiv \mathbf0$) goes back to Leray \cite{Leray1934}, for any divergence-free initial data $\bu_0\in L^2(\mathbb{R}^3)$; an analogous result for nonzero $\Bf$ can be found in  \cite[Chapter III]{Temam1979}. Leray also raised the question whether the $L^2$-norm of a weak solution tends to zero as $t\to\infty$. This was answered positively by Kato \cite{Kato1984} for strong solutions, and by Masuda \cite{Masuda1984} for weak solutions satisfying a strong energy inequality. Later, Schonbek \cite{Schonbek1985, Schonbek1986} established explicit decay rates for the $L^2$-norm using the Fourier splitting technique. 

For the MHD equations \cref{eq:IVPMHD}, several existence results have been established. In two spatial dimensions, when $\nu > 0$ and $\mu > 0$, global well-posedness for the Cauchy
problem has been obtained in  \cite{MR716200,MR346289} for bounded domains and in \cite{MR1616418} for the whole space. Using the  vorticity-current formulation, local-in-time existence and regularity was more recently obtained in \cite{zbMATH08005764} for $L^1$ initial data. In three dimensions, local-in-time well-posedness is contained as well in \cite{MR716200,MR346289,MR1616418} , while global-in-time well-posedness
for small initial data was later obtained in \cite{MR3780143}.\footnote{~For the inviscid MHD equations (i.\,e., $\nu = 0$ and $\mu > 0$), global well-posedness was
proved in \cite{MR1007099} in the two-dimensional setting. In three dimensions, local
well-posedness was established in \cite{MR2507450}, while global existence together with
local-in-time well-posedness in critical Besov spaces was obtained in \cite{MR4496608}.

In the case of vanishing magnetic resistivity (i.\,e., $\mu = 0$ and $\nu > 0$), local
existence and uniqueness in two-dimensional spatial domains were proved in
\cite{MR2265204,MR3217057,MR2772455}. Local existence of solutions in Besov spaces, in both
two and three dimensions, can be found in \cite{MR3415680,MR3590662}. Global well-posedness
in this regime has been achieved in \cite{MR3766969} for a special geometry, namely an
infinite slab, and for particular initial data given by a uniform non-parallel magnetic
field.

Finally, for the ideal MHD equations (i.\,e., $\mu = \nu = 0$), the local existence of strong
solutions was established in \cite{MR952901,MR1257135}; additional existence results can be found in \cite{zbMATH07327027,zbMATH05675031,zbMATH07635066,zbMATH08053923,zbMATH06383463}.}

The time-asymptotic behavior of solutions to the MHD equations has also been widely studied. In \cite{SSS1996}, lower and upper bounds are obtained for the decay rates of both the total and magnetic energies for solutions in $\mathbb{R}^n$, $2 \le n \le 4$. It is shown that weak solutions with large initial data outside the class of functions with radially equi-distributed energy exhibit algebraic decay, whereas radially equi-distributed initial data lead to exponential decay. In \cite{AgapitoSchonbek2007}, the long-time behavior of solutions in  $\mathbb{R}^n$, $2 \le n \le 3$, is investigated. In the absence of magnetic diffusion, it is proved that, provided strong bounded solutions exist, the total energy cannot exhibit asymptotic oscillations. When magnetic diffusion is present and the initial data belong only to $L^2$, solutions are shown to decay to zero without a uniform rate, and this lack of uniformity is optimal. Some decay results in exterior domains are available in \cite{MR894160}.

\subsection{Landau solutions for stationary Navier--Stokes and MHD equations}

In this work, we focus on the long-time behavior of weak solutions to the three-dimensional incompressible MHD system \cref{eq:IVPMHD} near a stationary \textit{Landau solution}.

The Landau solutions, introduced in \cite{Landau1944}, are stationary solutions to the Navier--Stokes equations that are homogeneous of degree $-1$ in $\mathbb{R}^3\setminus\{\mathbf0\}$, and are axisymmetric without swirl (see \cref{ssec:existence} for the definition and explicit formulas, or \cite[Section 8.2]{Tsai2018}).

Tian and Xin \cite{TianXin1998}, and Cannone and Karch \cite{CannoneKarch2004} proved that all $(-1)$-homogeneous, axisymmetric solutions in $C^2\left(\mathbb{R}^3 \setminus\{\mathbf0\}\right)$ are Landau solutions.  Later, in \cite{Sverak2011}, \v{S}ver\'{a}k proved the assertion without the assumption of axisymmetry. 

We stress that we focus on the three-dimensional case. Indeed, for $n = 2$, the problem of finding  $(-1)$-homogeneous solutions reduces to studying an ODE on the circle $\mathbb S^1$, where a reasonably complete description of solutions can be obtained in terms of elliptic functions; under the additional constraint that $\nabla \cdot \boldsymbol{u} = 0$ across the origin, there exists (modulo rotations) a countable family of $(-1)$-homogeneous solutions that are smooth away from the origin. On the other hand, for $n \ge 4$, there are no non-trivial $(-1)$-homogeneous solutions that are smooth away from the origin. We refer to \cite{Sverak2011} and the references therein for further details.

Landau solutions play a fundamental role in stationary Navier--Stokes flows: Korolev and \v{S}ver\'{a}k \cite{KorolevSverak2011} showed that the asymptotic leading term of small solutions in exterior domains of $\mathbb{R}^3$ must be a Landau solution; Miura and Tsai \cite{MiuraTsai2012} proved that the leading term of a point singularity like $|\bx|^{-1}$ at $\bx=\mathbf0$ is also given by a Landau solution when it is sufficiently small.

For the stationary MHD system, the pair $(\bu, \bB)$ is a $(-1)$-homogeneous, axisymmetric solution whenever $\bu$ is a Landau solution and $\bB \equiv \boldsymbol 0$. The converse direction has been investigated under different structural assumptions. For the system \textit{without magnetic diffusion} (i.\,e., $\mu=0$), Zhang \cite{Zhang2025MHDwithoutMD} showed that the converse holds provided the magnetic field $\bB$ has only the swirl component. In the case \textit{with magnetic diffusion} (i.\,e., $\mu >0$), Zhang, Wang, and Wang \cite{ZhangWangWang2026} proved that if $(\bu, \bB)$ is a smooth axisymmetric self-similar solution with $\bB = B^\theta(\rho, \varphi) \, \boldsymbol{e}_\theta$ and satisfies the pointwise decay condition $|\bu(\bx)| \leq \varepsilon |\bx|^{-1}$ for some small $\varepsilon > 0$, then $\boldsymbol{u}$ must be a Landau solution and $\boldsymbol{B}$ vanishes. 
This condition was later refined by Zhang \cite{Zhang2025MHDwithMD}, who showed that the conclusion holds under the weaker assumption that the cylindrical radial component $u^r \le \beta$ on the unit sphere for some constant $\beta < \frac{2\sqrt{2}}{3}$.

\subsection{Perturbations of the Landau solution}\label{ssec:perturb-setting}

For simplicity, let us fix $\nu=\mu=1$. Consider initial data of the form $\bu_0=\bU^\bb+\bw_0$ and $\bB_0$, and external force $\Bf=\bb\delta_{\mathbf{0}}$, where $\bU^\bb$ is the Landau solution corresponding to a vector field $\bb$ (as defined in \cref{ssec:existence}) and $\bw_0, \bB_0 \in L^2_{\sigma}(\mathbb{R}^3)$ (cf.~\cref{ssec:weakLpspace} for the notation), provided $|\bb|>0$ is sufficiently small.
Throughout this paper, $\delta_{\mathbf{0}}$ denotes the Dirac distribution, while $
\delta, \tilde \delta,  \delta_1,\delta_2,\dots$ are positive constants.


Let $(\bu(\bx,t), \bB(\bx,t), p(\bx,t))$ be a solution to the MHD equations \cref{eq:IVPMHD}. Setting $\bu(\bx,t)=\bU^{\bb}(\bx)+\bw(\bx,t)$ and $p(\bx,t)=P^\bb(\bx)+\pi(\bx,t)$, where $(\bU^\bb,P^\bb)$ is the Landau solution defined in \cref{ssec:existence},  we obtain the perturbed system
\begin{align}\label{eqn:PerturbedEqn}
\begin{cases}
\pt \bw + \bw \cdot \nabla \bw + \bU^\bb\cdot \nabla\bw+\bw\cdot \nabla \bU^\bb+\nabla \pi = \Delta\bw + \bB \cdot \nabla \bB, & (\bx,t)\in \mathbb{R}^3\times(0,\infty),\\
\pt \bB+\bw \cdot \nabla \bB+ \bU^\bb \cdot \nabla \bB=\Delta \bB + \bB \cdot \nabla \bw + \bB\cdot \nabla \bU^\bb, & (\bx,t)\in \mathbb{R}^3\times(0,\infty),\\
\nabla \cdot \bw=0, & (\bx,t)\in \mathbb{R}^3\times(0,\infty),\\
\nabla \cdot \bB=0, & (\bx,t)\in \mathbb{R}^3\times(0,\infty),\\
\bw(\bx,0)=\bw_0(\bx), & \bx \in \mathbb{R}^3,\\
\bB(\bx,0)=\bB_0(\bx), & \bx \in \mathbb{R}^3.
\end{cases}
\end{align}

When $\bB \equiv \mathbf{0}$, system \cref{eqn:PerturbedEqn} reduces to the perturbed Navier--Stokes equations around the Landau solution:
\begin{align}\label{eqn:PerturbedEqnNS}
\begin{cases}
\partial_t \bw + \bw \cdot \nabla \bw + \bU^\bb \cdot \nabla \bw + \bw \cdot \nabla \bU^\bb + \nabla \pi = \Delta \bw, & (\bx,t) \in \mathbb{R}^3 \times (0,\infty), \\
\nabla \cdot \bw = 0, & (\bx,t) \in \mathbb{R}^3 \times (0,\infty), \\
\bw(\bx,0) = \bw_0(\bx), & \bx \in \mathbb{R}^3.
\end{cases}
\end{align} 
In \cite{KarchPilarczyk2011}, Karch and Pilarczyk showed that if $|\bb|$ is sufficiently small and $\bw_0 \in L^2_\sigma(\mathbb{R}^3)$, there exists at least one weak solution $\bw$ to \eqref{eqn:PerturbedEqnNS} satisfying a strong energy inequality, and that any such solution is $L^2$--asymptotically stable, i.\,e., $\|\bw(t)\|_2 \to 0$ as $t \to \infty$. Under the additional assumption $\bw_0 \in L^q(\mathbb{R}^3)$ for some $\frac{6}{5} < q < 2$, the solution also obeys the decay estimate $\|\bw(t)\|_2 \leq C t^{-\frac{3}{2}(\frac{1}{q} - \frac{1}{2})}$. 
Subsequently, Karch, Pilarczyk, and Schonbek \cite{KarchPilarczykSchonbek17} generalized the $L^2$-stability of the Landau solution to a broader class of global-in-time solutions, including singular and time-dependent flows, provided the background solution is small enough.
The energy decay rate established in \cite{KarchPilarczyk2011} was further extended by Hishida and Schonbek \cite{HishidaSchonbek16}, who proved that for any $q\in [1,2)$, there exists a constant $\delta=\delta(q)>0$ such that if $|\bb|\le \delta$, the typical decay rate $t^{-\frac{3}{2}(\frac{1}{q}-\frac{1}{2})}$ also holds, thus extending the admissible range of $q$. 

\v{S}ver\'{a}k \cite{Sverak2011} proved that every $(-1)$-homogeneous non-zero solution of the stationary Navier--Stokes equations in $C^2(\mathbb{R}^3\setminus\{\mathbf0\})$ is a Landau solution.  
In \cite{LLY18I,LLY18II,LLY19III}, the authors studied $(-1)$-homogeneous axisymmetric solutions that are smooth away from the symmetry axis $\{x_1 = x_2 = 0\}$ and may be singular along the entire axis. In this class, all no-swirl solutions were classified in \cite{LLY18I,LLY18II}, while solutions with non-zero swirl were constructed in \cite{LLY18I,LLY19III}.  Among these solutions, the least singular ones (those with sufficiently mild singularities along the axis) were shown in \cite{LiYan2021} to be asymptotically stable under $L^2$-perturbations, thereby extending the stability result of \cite{KarchPilarczyk2011} beyond the Landau case.  

In \cite{LZZ23}, the $L^3$-asymptotic stability of Landau solutions was studied:  if initial perturbation $\boldsymbol{w}_0$ and  the parameter $|\bb|$ of the background Landau solution are sufficiently small, then there exists a unique global solution with $\|\boldsymbol{w}(t)\|_{3} \to 0$ as $t \to \infty$.   On the other hand, in \cite{zbMATH08070132}, the authors showed that the asymptotic stability around Landau solutions may fail for some initial perturbations in $L^{3,\infty}(\R^3)$ (regardless of how small the Landau solution or the initial perturbations are in $L^{3,\infty}(\R^3))$. 

Finally, in \cite{zbMATH07471753}, the authors studied the stability and asymptotic behavior of singular solutions to the 3D incompressible Navier--Stokes equations with singular external forces which are either singular finite measures or more general tempered distributions
with bounded Fourier transforms (including, in particular, the Landau solution).

We now define the notion of a weak solution to \cref{eqn:PerturbedEqn}.
In the following, $\langle \cdot,\cdot \rangle$ denotes the standard $L^2$-inner product, and
\begin{equation*}
\begin{aligned}
a_1(\bw,\bv) &= \langle \nabla \bw , \nabla \bv \rangle 
+ \langle (\bU^\bb \cdot \nabla) \bw , \bv \rangle + \langle (\bw \cdot \nabla) \bU^\bb, \bv \rangle,\\ 
a_2(\bB,\bv) &= \langle \nabla \bB , \nabla \bv \rangle 
+ \langle (\bU^\bb \cdot \nabla) \bB , \bv \rangle - \langle (\bB \cdot \nabla) \bU^\bb, \bv \rangle.
\end{aligned}
\end{equation*}
For initial data $\bw_0, \bB_0\in L^2_{\sigma}(\R^3)$, $(\bw,\bB)=(\bw(\bx,t),\bB(\bx,t))$ is called a \textit{weak solution} if $\bw,\bB\in X$, where
\[
X\coloneqq C_w \big( [0,\infty); L^2_\sigma (\R^3)\big)\cap L^2\big([0,\infty); \dot H_\sigma^1(\R^3)\big),
\]
and if for all $t\geq s\geq 0$ and all test functions $\boldsymbol{\vp}$ belonging to
\begin{align}\label{eqn:testfcnspace}
\boldsymbol{\vp} \in C\big( [0, \infty); L^2_\sigma(\R^3) \cap L^3(\R^3) \big), \quad 
\nabla \boldsymbol{\vp} \in L^2_{\mathrm{loc}}\big( [0,\infty); L^2(\R^3) \big), \quad 
\partial_\tau \boldsymbol{\vp} \in L^2_{\mathrm{loc}}\big( [0,\infty); (H^1_\sigma(\R^3))^* \big),
\end{align}
the following identities hold:
\begin{equation}\label{eqn:weak}
\begin{aligned}
&\langle \bw(t), \boldsymbol{\varphi}(t)  \rangle
+\int_s^t \Big[ 
\langle \bw\cdot  \nabla \bw, \boldsymbol{\varphi} \rangle
- \langle \bB\cdot  \nabla \bB, \boldsymbol{\varphi} \rangle
+ a_1(\bw,\boldsymbol{\vp})
\Big]\, \mathrm d\tau 
\\ &\qquad= \langle \bw(s), \boldsymbol{\varphi}(s)  \rangle 
+ \int_s^t \langle \bw, \boldsymbol{\varphi}_\tau \rangle \, \mathrm d\tau,\\
&\langle \bB(t), \boldsymbol{\varphi}(t)  \rangle
+\int_s^t \Big[ 
\langle \bw\cdot  \nabla \bB, \boldsymbol{\varphi} \rangle
-\langle \bB\cdot  \nabla \bw, \boldsymbol{\varphi} \rangle
+ a_2(\bB,\boldsymbol{\vp})
 \Big]\, \mathrm d\tau 
 \\ &\qquad= \langle \bB(s), \boldsymbol{\varphi}(s)  \rangle 
 + \int_s^t \langle \bB, \boldsymbol{\varphi}_\tau \rangle \, \mathrm d\tau.
\end{aligned}
\end{equation}
The test function space is a special case of the more general framework introduced in \cite[eq.~(2.16)]{HishidaSchonbek16}.

\subsection{Main results}\label{ssec:main}
This work investigates the long-time behavior and decay rates of weak solutions to \cref{eqn:PerturbedEqn} that obey a strong energy inequality of the form
\begin{align}\label{eqn:energyineq1}
\begin{aligned}
&\|\bw(t)\|_2^2 + \|\bB(t)\|_2^2 + 2 \int_s^t \bigl( \| \nabla \bw(\tau)\|_2^2 + \| \nabla \bB(\tau)\|_2^2 \bigr) \,\mathrm d \tau\\
& \le \|\bw(s)\|_2^2 + \|\bB(s)\|_2^2 - 2 \int_s^t \bigl( \langle (\bw \cdot \nabla) \bU^\bb, \bw \rangle - \langle (\bB \cdot \nabla) \bU^\bb, \bB \rangle \bigr) \,\mathrm d \tau,
\end{aligned}
\end{align}
which holds for almost every $s \ge 0$ (including $s=0$) and all $t \ge s$. 
Following the retarded mollification method in \cite[Appendix]{CKN82} (see also \cite[Theorem 3.9]{Tsai2018}), together with the Hardy-type inequality \eqref{eqn:Hardy}, one can show the existence of at least one weak solution to \cref{eqn:PerturbedEqn} satisfying \cref{eqn:energyineq1}, provided $|\bb|$ is small. The proof is given in the appendix.

Our first main result is as follows.
\begin{theorem}[Asymptotic stability]\label{thm:stability}
There exists a constant $\delta>0$ such that if $0<|\bb|\le \delta$, then every weak solution $(\bw(t),\bB(t))$ to \cref{eqn:PerturbedEqn} with \cref{eqn:energyineq1} satisfies
\begin{align*}
    \lim\limits_{t\to \infty} \|\bw(t)\|_2=\lim\limits_{t\to \infty} \|\bB(t)\|_2=0,
\end{align*}
provided $\bw_0, \bB_0 \in L^2_{\sigma}(\R^3)$.
\end{theorem}

Under additional integrability assumptions on initial data, we also get the decay rate of $\|\bw(t)\|_2$ and $\|\bB(t)\|_2$.

\begin{theorem}[Energy decay rate]\label{thm:decay}
    Let $1\le q<2$. There exists a constant $\tilde\delta=\tilde\delta(q)>0$ such that if $0<|\bb|\le \tilde\delta$, then every weak solution $(\bw(t),\bB(t))$ to \cref{eqn:PerturbedEqn} with \cref{eqn:energyineq1}  satisfies
\begin{align*}
    \|\bw(t)\|_2 \le C(t+1)^{-\frac{3}{2}(\frac{1}{q}-\frac{1}{2})},\quad 
    \|\bB(t)\|_2 \le C(t+1)^{-\frac{3}{2}(\frac{1}{q}-\frac{1}{2})},
\end{align*}
provided $\bw_0, \bB_0 \in L^q(\R^3)\cap L^2_{\sigma}(\R^3)$. Here $C$ is a constant depending only on $q$, the $L^q(\R^3)\cap L^2_\sigma(\R^3)$-norms of $\bw_0$ and $\bB_0$, and $\tilde{\delta}$.

\end{theorem}

These results extend the analysis of decay properties for weak solutions of \cref{eqn:PerturbedEqnNS} satisfying a strong energy inequality analogous to \cref{eqn:energyineq1}, previously carried out in \cite{HishidaSchonbek16}. Unlike \cite{HishidaSchonbek16}, we do not use the Fourier splitting method to obtain the energy decay; instead, we present a streamlined argument that combines ideas from \cite{HishidaSchonbek16}, \cite{KarchPilarczyk2011}, and \cite{BorchersMiyakawa92}. In particular, inspired by \cite{KarchPilarczyk2011}, we start from estimates of the type 
\begin{align*}
\|\bw(s)\|_2 
&\leq \|e^{-sL_1} \bw_0\|_2 
+ \int_0^s \bigl\|e^{-(s-\tau)L_1} \mathbb{P} \nabla\cdot (\bw \otimes \bw - \bB \otimes \bB)\bigr\|_2 \, \mathrm{d}\tau, \\
\|\bB(s)\|_2 
&\leq \|e^{-sL_2} \bB_0\|_2 
+ \int_0^s \bigl\|e^{-(s-\tau)L_2} \mathbb{P} \nabla\cdot (\bw \otimes \bB - \bB \otimes \bw)\bigr\|_2 \, \mathrm{d}\tau,
\end{align*}
(cf.~\crefrange{eqn:estimate2-1}{eqn:estimate2-2} below) and use them to derive the energy decay of solutions. The main difference from \cite{KarchPilarczyk2011} lies in our treatment of the linear and nonlinear terms: we employ the more flexible $L^q$--$L^2$ estimates from \cite{HishidaSchonbek16} (see, ~\crefrange{eqn:decay1}{eqn:decay2}), which yield a more direct proof.

\subsection{Outline} 
The paper is organized as follows. \cref{sec:prelim} collects some preliminary lemmas that will be used. 
In particular, in \cref{ssec:existence}, we recall the explicit form of the Landau solutions to the stationary Navier--Stokes equations and state Hardy's inequality; in \cref{ssec:weakLpspace}, we introduce the weak-$L^p$ (Marcinkiewicz) spaces and present the weak versions of Young's and H\"older's inequalities.
In \cref{sec:linear}, we study the linearized perturbed system, proving that the linear operators generate analytic semigroups and stating the decay rate of the evolution operator. Finally, the proofs of \crefrange{thm:stability}{thm:decay} are carried out in \cref{sec:proofs}.
In the appendix, we use the retarded mollification method to construct weak solutions satisfying the strong energy inequality (see \cref{thm:MHD-existence}).

\section{Preliminaries}
\label{sec:prelim}

\subsection{Landau solutions and Hardy-type inequality}
\label{ssec:existence}

This subsection is devoted to recalling the Landau solutions to the stationary Navier--Stokes equations and Hardy's inequality.

We begin with some basic facts about the Landau solutions (see e.\,g., \cite[Section 8.2]{Tsai2018}). For each nonzero $\bb \in \mathbb{R}^3$, there exists a unique Landau solution $\bU^\bb$ and an associated pressure $P^\bb$, homogeneous of degrees $-1$ and $-2$ respectively, that solve the stationary Navier--Stokes equations
\begin{align}\label{eq:NS}
\begin{cases}
\bU^\bb \cdot \nabla \bU^\bb + \nabla P^\bb = \Delta \bU^\bb, & \bx \in \R^3\setminus \{\mathbf0\}, \\
\nabla \cdot \bU^\bb = 0, & \bx \in \R^3\setminus \{\mathbf0\},
\end{cases}
\end{align}
and are smooth in $\mathbb{R}^3 \setminus \{\mathbf0\}$. Moreover, in the sense of distributions in $\mathbb{R}^3$, they satisfy
\begin{align*}
-\Delta \bU^\bb + \nabla \cdot (\bU^\bb \otimes \bU^\bb) + \nabla P^\bb = \bb \, \delta_{\mathbf{0}}, \qquad 
\nabla \cdot  \bU^\bb = 0,
\end{align*}
where $\delta_{\mathbf{0}}$ denotes the Dirac distribution.

When $\bb = (0,0,\beta)$ with $\beta>0$, the Landau solution admits an explicit representation in spherical coordinates. Let $\bx = (\rho\sin\varphi\cos\theta,\rho\sin\varphi\sin\theta, \rho\cos\varphi)$ and define the orthonormal basis vectors
\[
\be_{\rho}=\frac{\bx}{\rho}, \quad \be_{\theta}=(-\sin\theta, \cos\theta,0), \quad \be_{\varphi}=\be_{\theta}\times \be_{\rho}.
\]
Then for any $A>1$,
\begin{align}\label{eqn:Landausolution}
    \bU^\bb =\frac{2}{\rho} 
    \left[
    \left(\frac{A^2-1}{(A-\cos\varphi)^2}-1\right)\be_{\rho}
    +\frac{-\sin\varphi}{A-\cos\varphi}\be_{\varphi}
    \right], \quad
    P^\bb =\frac{4(A \cos\varphi-1)}{\rho^2 (A-\cos\varphi)^2},
\end{align}
where the parameter $\beta$ is related to $A$ by
\[
\beta = 16 \pi 
\left[A + \frac 12 A^2 \log \left(\frac {A-1}{A+1}\right) + \frac{4A}{3(A^2-1)}\right].
\]
The corresponding Landau solution for a general $\bb$ is obtained by rotation.

There exist $\ve_0>0$ and $C_0>0$ such that for any $\bb \in \R^3$ with $0 < |\bb| < \ve_0$, 
\begin{align}\label{eqn:LandauEst}
    |\bU^\bb(\bx)|\le\frac{C_0|\bb|}{|\bx|},\quad |\nabla \bU^\bb(\bx)|\le\frac{C_0|\bb|}{|\bx|^2},\quad \text{in }\R^3\setminus\{\mathbf0\}.
\end{align}
Indeed, a direct computation shows that for $\bb=(0,0,\beta)$ with sufficiently small $|\beta|>0$,
\[
|\bU^\bb(\bx)| \lesssim \frac{1}{|A|} \frac{1}{|\bx|},\quad
|\nabla \bU^\bb(\bx)| \lesssim \frac{1}{|A|} \frac{1}{|\bx|^2}.
\]
Combining this with $\frac{\beta(A)}{16 \pi}=\frac{1}{A}+o(\frac{1}{A})$ as $A\to \infty$ yields \eqref{eqn:LandauEst}.

A vector field $\bv$ is called \emph{axisymmetric} if it is of the form
\[
\bv=v^\rho(\rho,\vp) \be_\rho + v^\theta(\rho,\vp) \be_\theta + v^\vp(\rho,\vp) \be_\vp,
\]
with $v^\rho, v^\theta, v^\vp$ independent of $\theta$. 
It is called \emph{without swirl} if $v^\theta = 0$.

Finally, we recall the following \textit{Hardy-type inequality}:
\begin{align}\label{eqn:Hardy}
    \int_{\R^3} \frac{|w(\bx)|^2}{|\bx|^2}\, \mathrm{d}x \le 4 \int_{\R^3} |\nabla w(\bx)|^2 \, \mathrm{d}x \quad \text{for all } w\in \dot{H}^1(\R^3),
\end{align}
a proof of which is given in \cite[Chapter I.6]{Leray1934}.\footnote{~We remark that \cref{eqn:Hardy} is valid when $n\ge 3$, with the constant $4$ replaced by $\frac{4}{(n-2)^2}$  (see for example \cite[Theorem 2.1]{AGG2006}); this restriction constitutes a technical difficulty in the analysis of the stability of $(-1)$-homogeneous solutions in $\mathbb R^2$.}

\subsection{Notations and inequalities in weak \texorpdfstring{$L^p$}{Lp} spaces}
\label{ssec:weakLpspace}

The following notation will be used throughout the paper. Let $L^p_\sigma(\mathbb{R}^3)$ denote the completion of $C_{c,\sigma}^{\infty}(\mathbb{R}^3)$ in the $L^p$-norm $\|\cdot\|_{L^p(\mathbb{R}^3)}$, where $C_{c,\sigma}^{\infty}(\mathbb{R}^3)$ is the space of smooth, compactly supported, divergence-free vector fields.
In the same way, $H^1_\sigma(\mathbb{R}^3)$ and $\dot H^1_\sigma(\mathbb{R}^3)$ are the completions of $C_{c,\sigma}^{\infty}(\mathbb{R}^3)$ with respect to $\|\cdot\|_{H^1(\mathbb{R}^3)}$ and $\|\nabla\cdot\|_{L^2(\mathbb{R}^3)}$, 
respectively.

For $1<p<\infty$, let $L^{p,\infty}=L^{p,\infty}(\mathbb{R})$ denote the \textit{weak Marcinkiewicz space}, i.\,e., Banach space of measurable functions on $\mathbb{R}$ equipped with the norm
\begin{align}\label{eqn:weakLpnorm}
\|f\|_{p,\infty}= \sup
\left\{|E|^{-1+\frac{1}{p}} \int_E |f(s)| \, \mathrm{d}s: E \in \mathcal{B}\right\}.
\end{align}
Here $|E|$ denotes the Lebesgue measure of a measurable set $E$, and $\mathcal{B}$ is the collection of all Borel sets with finite positive measure.

One can verify that $|\cdot|^{-\frac{1}{p}} \in L^{p,\infty}(0,\infty)$ for $p>1$. This follows from the Hardy--Littlewood inequality $\int_E s^{-\frac{1}{p}} \, \mathrm{d}s \le \int_0^{|E|} s^{-\frac{1}{p}} \, \mathrm{d}s$ together with \cref{eqn:weakLpnorm}.
Recall also the well-known embedding $L^p(\R) \subset L^{p,\infty}(\R)$, which is a consequence of H\"older's inequality and \cref{eqn:weakLpnorm}.

The following lemma provides weak versions of Young's and H\"older's inequalities (see  \cite[Lemma 2.2]{BorchersMiyakawa92} for a detailed proof).

\begin{lemma} 
Let $f\in L^{p,\infty}(\R)$ and $g\in L^{q,\infty}(\R)$ with $1<p,q<\infty$.
{
\renewcommand{\descriptionlabel}[1]{\emph{#1}}
\begin{description}
\item[Weak Young's inequality:] If $1+\frac{1}{r}=\frac{1}{p}+\frac{1}{q}$ with $1<r<\infty$, then the convolution $f\ast g$ belongs to $L^{r,\infty}(\R)$ and satisfies
\begin{align}\label{eqn:Youngineq}
\|f \ast g\|_{r,\infty} \le C_{p,q} \|f\|_{p,\infty} \|g\|_{q,\infty},
\end{align}
where $C_{p,q}>0$ is a constant depending only on $p$ and $q$.

\item[Weak H\"older's inequality:] If $\frac{1}{r}=\frac{1}{p}+\frac{1}{q}$ with $1<r<\infty$, then the pointwise product $fg$ belongs to $L^{r,\infty}(\R)$ and satisfies
\begin{align}\label{eqn:Holderineq}
\|fg\|_{r,\infty} \le C_{p,q} \|f\|_{p,\infty} \|g\|_{q,\infty},
\end{align}
where $C_{p,q}>0$ is a constant depending only on $p$ and $q$.
\end{description}
}
\end{lemma}

\section{Linearized flow}
\label{sec:linear}

In this section, we consider the following linearized perturbed equations:
\begin{align}\label{eqn:LPE}
\begin{cases}
\pt \bw - \Delta\bw +  \bU^\bb\cdot \nabla\bw+\bw\cdot \nabla \bU^\bb+\nabla \pi = 0, 
&(\bx,t)\in \mathbb{R}^3\times(0,\infty),\\
\pt \bB - \Delta \bB + \bU^\bb \cdot \nabla \bB-\bB\cdot \nabla \bU^\bb =0,
& (\bx,t)\in \mathbb{R}^3\times(0,\infty),\\
\nabla \cdot \bw =0, 
& (\bx,t)\in \mathbb{R}^3\times(0,\infty),\\
\nabla \cdot \bB=0, & (\bx,t)\in \mathbb{R}^3\times(0,\infty),\\
\bw(\bx,0)=\bw_0(\bx), & \bx \in \mathbb{R}^3,\\
\bB(\bx,0)=\bB_0(\bx), & \bx \in \mathbb{R}^3.
\end{cases}
\end{align}
The \emph{Leray projector} $\mathbb P$ is given by $(\delta_{jk}+R_jR_k)_{1\le j,k\le 3}$, where $R_j$ signifies the \emph{Riesz transform}. When $1<q<\infty$, this operator remains bounded on $L^q(\R^3)$ and acts as a projection onto $L^q_{\sigma}(\R^3)$, corresponding to the \textit{Helmholtz decomposition} $L^q(\R^3) = L^q_{\sigma}(\R^3) \oplus \{\nabla p : \ p \in L^q_{\loc}(\R^3)\}$ for $1<q<\infty$ (see, e.\,g., \cite[Chapter III]{Galdi11} or \cite[Section 1]{Tsai2018}).

Applying the Leray projector $\mathbb P$ to \cref{eqn:LPE} formally yields
\begin{align}\label{eqn:LPE2}
\begin{cases}
\p_t \bw + L_1 \bw =0, & t \in   (0,\infty), \\ \bw(0)=\bw_0,\\
\p_t \bB + L_2 \bB =0, & t \in (0,\infty), \\ \bB(0)=\bB_0,
\end{cases}   
\end{align}
where 
\begin{align}\label{eqn:linearoperator}
\begin{aligned}
L_1 \bw&\coloneqq \mathbb P \bigl[- \Delta\bw +  \bU^\bb\cdot \nabla\bw + \nabla\cdot  (\bw \otimes \bU^\bb)\bigr],\\
L_2 \bB&\coloneqq \mathbb P \bigl[- \Delta \bB + \bU^\bb \cdot \nabla \bB-\nabla\cdot (\bB \otimes \bU^\bb )\bigr].
\end{aligned}
\end{align}

Denote the adjoint linear operators by
\begin{align}\label{eqn:adjlinoperator}
\begin{aligned}
L_1^*\bv &=\mathbb P \bigl[-\Delta \bv-(\bU^\bb\cdot \nabla)\bv-\sum_{j=1}^3 U_j^\bb \nabla v_j\bigr],\\
L_2^*\bv &=\mathbb P \bigl[-\Delta \bv-(\bU^\bb\cdot \nabla)\bv+\sum_{j=1}^3 U_j^\bb \nabla v_j\bigr],
\end{aligned}
\end{align}
and the corresponding bilinear forms are
\begin{equation}
\begin{aligned}
\label{eqn:bilinearform}
a_1(\bw,\bv)&=\langle \nabla \bw , \nabla \bv \rangle 
    + \langle (\bU^\bb \cdot \nabla) \bw , \bv \rangle 
    +\langle (\bw \cdot \nabla) \bU^\bb, \bv \rangle,\\ 
a_2(\bB,\bv)&=\langle \nabla \bB,\nabla \bv \rangle 
    + \langle (\bU^\bb \cdot \nabla) \bB , \bv \rangle 
    -\langle (\bB \cdot \nabla) \bU^\bb, \bv \rangle.    
\end{aligned}
\end{equation}
Then
\[
a_1(\bw,\bv)=\langle L_1 \bw,\bv \rangle =\langle  \bw, L_1^*\bv \rangle,\quad 
a_2(\bB,\bv)=\langle L_2 \bB,\bv \rangle =\langle  \bB, L_2^*\bv\rangle.
\]

Let us recall a criterion due to Lions (see \cite[Proposition 1.2, Remarque 1.1]{Lions61}), which gives sufficient conditions for an operator to generate a holomorphic semigroup on a Hilbert space.

\begin{proposition}\label{prop:Lions}
Let $\mathcal{H}$ be a Hilbert space and let $\mathcal{V}\subset \mathcal{H}$ be a dense subspace. Assume that $\mathcal{V}$ is a Hilbert space with the inner product $(\cdot , \cdot )_{\mathcal{V}}$ and with the norm $\| \cdot \|_{\mathcal{V}}$ such that for a constant $C>0$ we have $ \|x\|_{\mathcal{H}}\leqslant C\|x\|_{\mathcal{V}}$ for all $x\in \mathcal{V}$. Let $a(x,y)$ be a bounded sesquilinear form on $\mathcal{V}$, which defines an operator $A: \mathcal{D}(A)\rightarrow \mathcal{H}$ as follows
\begin{equation*}
   \mathcal{D}(A)=\{ z\in \mathcal{V} :\ |a(z,v)|\leqslant C \| v\|_{\mathcal{H}}, v\in \mathcal{V}\}, \quad (Az,v)_{\mathcal{H}}=a(z,v).
\end{equation*}
Suppose that for some $\alpha >0$ and $\lambda_0 \in \mathbb R$ we have 
\begin{equation}\label{eqn:Lions}
  \alpha \| z\|_{\mathcal{V}}^2\leqslant \text{\rm Re} \ a(z,z)+\lambda_0\|z\|_{\mathcal{H}}^2.
\end{equation}
Then $-A$ is the infinitesimal generator of a strongly continuous semigroup of linear operators on $\mathcal{H}$ which is holomorphic in a sector $S_\ve\coloneqq \{s\in \mathbb C: \ |\text{\rm Arg}\ s|<\ve\}$ for some $\ve >0$.
\end{proposition}

The next theorem follows directly from \cref{prop:Lions} and the Hardy inequality \eqref{eqn:Hardy}. It establishes that for sufficiently small $|\bb|$, the operators $-L_i,\,-L_i^*$ $(i=1,2)$ are the infinitesimal generators of analytic semigroups of linear operators on $L^2_{\sigma}(\mathbb R^3)$.

\begin{theorem}\label{thm:analyticsemigroup}
Let $\mathbf0\neq \bb \in \R^3$ satisfy $|\bb|< \min\{\ve_0,\frac{1}{4C_0}\}$, where $C_0$ is defined in \cref{eqn:LandauEst}. Then the operators $-L_i$ and $-L_i^*$ (for $i=1,2$)  defined in \cref{eqn:linearoperator} and \cref{eqn:adjlinoperator} are the infinitesimal generators of analytic semigroups of linear operators on $L^2_{\sigma}(\mathbb R^3)$ which are holomorphic in a sector $ \{s\in \mathbb C: \ |\text{\rm Arg}\ s|<\ve\}$ for some $\ve=\ve(|\bb|) >0$.
\end{theorem}

\begin{proof}
The argument follows that of \cite[Theorem 4.2]{KarchPilarczyk2011}; we outline the details for completeness.
Take $\mathcal{H}=L^2_\sigma (\R^3)$ and $\mathcal{V}=H^1_\sigma (\R^3)$ in \cref{prop:Lions}. Then \cref{eqn:LandauEst} and Hardy's inequality \cref{eqn:Hardy} yield that
\[
|\langle (\bU^\bb \cdot \nabla) \bw , \bv \rangle| \le C_0  |\bb| \| \nabla \bw \|_2 \left\| \frac{\bv}{|\cdot|} \right\|_2
\le 2 C_0  |\bb| \| \nabla \bw \|_2 \| \nabla \bv \|_2,
\]
and 
\[
|\langle (\bw \cdot \nabla) \bU^\bb, \bv \rangle| \le C_0 |\bb| \left\| \frac{\bw}{|\cdot|} \right\|_2 
\left\| \frac{\bv}{|\cdot|} \right\|_2 
\le 4C_0 |\bb|  \| \nabla \bw \|_2 \| \nabla \bv \|_2,
\]
provided $|\bb| \le \ve_0$.
Consequently,
\begin{equation}
\begin{aligned}\label{eqn:bddcoercive}
a_1(\bw,\bv) &\le (1+6C_0|\bb|) \|\nabla \bw \|_2 \|\nabla \bv \|_2,\\ 
a_1(\bw,\bw) &\ge (1-4C_0|\bb|) \|\nabla \bw \|_2^2.
\end{aligned}   
\end{equation}
Here for the lower bound estimate we also used the cancellation
$\langle \bU^\bb\cdot\nabla \bw,\bw\rangle=0$, which follows from $\nabla\cdot \bU^\bb=0$ via a standard approximation argument.
Hence \cref{eqn:Lions} is satisfied provided $|\bb| < \min \{\ve_0,\frac{1}{4C_0}\}$, which gives the desired conclusion for $-L_1$ and $-L_1^*$. The same argument applies to $-L_2$ and $-L_2^*$.
\end{proof}

Using the coercivity and continuity estimates for the bilinear forms
$a_1,a_2$ established in the proof of \cref{thm:analyticsemigroup} (see \cref{eqn:bddcoercive}), the
following proposition follows by the standard Galerkin procedure.

\begin{proposition}\label{prop:reg} 
Assume $0<|\bb|< \min\{\ve_0, \frac{1}{4C_0} \}$ and let $i\in\{1,2\}$.  
For every $\boldsymbol{\psi}\in L^2_\sigma(\R^3)$ and $T>0$, the solution $\bv(\tau)=e^{-\tau L_i^*}\boldsymbol{\psi}$ satisfies
\[
\bv\in C([0,T];L^2_\sigma(\mathbb R^3))\cap L^2(0,T;H^1_\sigma(\mathbb R^3)), \qquad
\partial_\tau \bv\in L^2(0,T;(H^1_\sigma(\R^3))^*).
\]
\end{proposition}

\begin{proof}
The argument is completely parallel to the proof of Lemma~\ref{apdx-mhd-lm}. We only sketch it for $L_1^*$; the case of $L_2^*$ is identical.

Let $\{\bg_j\}\subset C^\infty_{c,\sigma}(\R^3)$ be an orthonormal basis of
\(L^2_\sigma(\R^3)\) whose span is dense in $H^1_\sigma(\R^3)$. Define the
Galerkin approximations $\bv_m=\sum_{j=1}^m d_{jm}\bg_j$ by
\begin{align*}
\begin{aligned}
&\langle \partial_\tau \bv_m,\bg_j\rangle + a_1(\bg_j, \bv_m)=0,\qquad j=1,\cdots,m,
\end{aligned}
\end{align*}
with $\bv_m(0)=\sum_{j=1}^m \langle \boldsymbol{\psi}, \bg_j \rangle \bg_j$, 
where $a_1$ is the bilinear form associated with $L_1$.
Since the test functions $\bg_j$ are divergence-free, we have $\mathbb P\bg_j=\bg_j$. Thus, when testing the equation with $\bg_j$, the Leray projection $\mathbb P$ does not appear.
Using $\nabla\cdot \bU^\bb=0$ and the Hardy estimates \cref{eqn:bddcoercive}, the same energy computation as in Lemma~\ref{apdx-mhd-lm} yields
\[
\|\bv_m(t)\|_2^2 + 2(1-4C_0|\bb|) \int_0^t \|\nabla \bv_m(s)\|_2^2\, \mathrm{d}s \le \|\bv_m(0)\|_2^2 \le \|\boldsymbol\psi\|_2^2.
\]
Consequently, uniformly in $m$,
\[
\|\bv_m\|_{L^\infty(0,T;L^2_\sigma(\R^3))}
+
\|\bv_m\|_{L^2(0,T;H^1_{\sigma}(\R^3))}
\le C\|\boldsymbol\psi\|_2.
\]
Passing to a subsequence, the weak limit $\bv$ satisfies $\bv\in L^\infty(0,T;L^2_\sigma(\R^3))\cap L^2(0,T;H^1_\sigma(\R^3))$.
Testing the weak form with $\boldsymbol\chi\in H^1_\sigma(\R^3)$ gives, as in Step 4 of Lemma~\ref{apdx-mhd-lm},
\[
|\langle \partial_\tau \bv,\boldsymbol\chi\rangle|
\le C\|\nabla \bv\|_2\|\boldsymbol\chi\|_{H^1},
\]
where $C$ is independent of $t$.
Hence $\partial_\tau \bv\in L^2(0,T;(H^1_\sigma(\R^3))^*)$ follows from a duality argument and $\nabla \bv \in L^2(0,T;L^2(\R^3))$. 
Finally, by \cite[Lemma 3.3(iii)]{Tsai2018}, after redefinition on a null set,
\(\bv\in C([0,T];L^2_\sigma(\R^3))\).
\end{proof}

\begin{remark}
If the initial data enjoy higher regularity, namely
$\boldsymbol\psi\in H^1_\sigma(\R^3)$, then the solution $\bv$ in
Proposition~\ref{prop:reg} has improved regularity. More precisely, using the
gradient energy estimate from \cite[Proposition 3.3]{HishidaSchonbek16}, under
the same smallness condition on $|\bb|$, we obtain
\[
\bv\in L^\infty(0,T;H^1_\sigma(\R^3))\cap L^2(0,T;H^2(\R^3)),
\qquad
\partial_\tau\bv\in L^2(0,T;L^2_\sigma(\R^3)).
\]
\end{remark}

The decay estimates below for the semigroups $e^{-tL_i}$ were first established by Hishida and Schonbek in  \cite{HishidaSchonbek16}, and play a key role in the proofs of \crefrange{thm:stability}{thm:decay}. Indeed, the linearization decouples the system \cref{eqn:LPE2}, enabling us to argue similarly to  \cite{HishidaSchonbek16}. 
Their proof is based on the Duhamel formulation of the linearized equation. Treating the term $\bu \otimes \bU^\bb + \bU^\bb \otimes \bu$ as a perturbation and applying the key estimate by Yamazaki \cite[Corollary 2.3(2)]{MR1777114} in Lorentz spaces, they set up a contraction mapping in an appropriately weighted $L^p$ space, under the condition that the basic flow $\bU^\bb$ is sufficiently small in the critical space $L^{3,\infty}(\R^3)$.

\begin{proposition}\label{prop:decay} 
Let $L_i$, with $i =1,\,2$, be the linear operators defined in \cref{eqn:linearoperator}. Then the following estimates hold.
\begin{propenum}
\item \label{it:lmdecay1} For any $1 \le q \le 2$, there exists a constant $\delta_1=\delta_1(q)>0$ (independent of $i$) such that if $0<|\bb|\le \delta_1$, then the semigroup $e^{-tL_i}$ satisfies
\begin{align}\label{eqn:decay1}
\|e^{-tL_i} \Bf \|_2 \le Ct^{-\frac{3}{2}(\frac{1}{q}-\frac{1}{2})} \|\Bf\|_q,\quad \text{for all } t>0, \ \Bf \in  L^q_{\sigma}(\mathbb R^3). 
\end{align}

\item \label{it:lmdecay2}
For any $\frac{6}{5} < q \le 2$, there exists $\delta_2=\delta_2(q)>0$ (independent of $i$) such that if $0<|\bb|\le \delta_2$, then the semigroup $e^{-tL_i}$ satisfies
\begin{align}\label{eqn:decay2}
\|e^{-tL_i}\mathbb P \nabla\cdot \BF\|_2 
\le Ct^{ -\frac{3}{2}(\frac{1}{q}-\frac{1}{2}) -\frac{1}{2} } \| \BF \|_q,\quad \text{for all } t>0, \  \BF\in L^q(\mathbb R^3)^{3\times 3}. 
\end{align}
\end{propenum}
\end{proposition}

\begin{proof} We prove the two claims by relying on the results in \cite{HishidaSchonbek16}. 
The operators $L_1$ and $L_2$ differ only in the sign of one term involving $\boldsymbol{U}^{\boldsymbol{b}}$; however, the proofs in \cite{HishidaSchonbek16} rely solely on the divergence-free condition, the smallness of the background flow in $L^\infty(0,\infty; L^{n,\infty}(\R^n))$, and the weak-$*$ time continuity of the background flow, all of which remain valid regardless of the sign of the individual terms involving $\bU^{\bb}$. Therefore, the same arguments apply directly to $L_2$.
\begin{enumerate}[label={(\roman*)}]
    \item Since equations \eqref{eqn:LPE2}$_1$ and \eqref{eqn:linearoperator} coincide with eq.~(2.7) in \cite{HishidaSchonbek16}, the decay estimate for $e^{-tL_1}$ follows directly from \cite[Theorem 2.2]{HishidaSchonbek16} by choosing $r_0 = 3$ and $r = 2$. The proof for $e^{-tL_2}$ is analogous and will be omitted.
\item The estimate \cref{eqn:decay2} is implicitly contained in the proof of  \cite[Proposition~4.6]{HishidaSchonbek16}.
For any $1<q<\infty$, fix an exponent $r_2$ satisfying  
\[
\max\{q,\tfrac32\}<r_2<\infty,\qquad \frac1q-\frac1{r_2}<\frac13.
\] 
Such an $r_2$ exists and is determined by $q$. From the proof of Proposition~4.6 in \cite{HishidaSchonbek16}, there exists a constant $\bar \delta=\bar \delta(q)>0$ such that if $|\bb|\le \bar\delta$, then the following estimates hold for all $t>0$ and $\BF\in L^q(\mathbb R^3)^{3\times3}$:
\begin{align*}
\|e^{-tL_1}\mathbb P\operatorname{div}\BF\|_q &\le C t^{-1/2}\|\BF\|_q,\\
\|e^{-tL_1}\mathbb P\operatorname{div}\BF\|_{r_2,\infty} &\le C t^{-\frac32\bigl(\frac1q-\frac1{r_2}\bigr)-\frac12}\|\BF\|_q.
\end{align*}
One can verify that $q<2<r_2$ provided $\frac65<q<2$.
By real interpolation of Lorentz spaces (see, e.\,g., \cite[Theorem 5.3.1]{BerghLofstrom76}),
\[
(L^{q,\infty}(\R^3), L^{r_2,\infty}(\R^3))_{\theta,2}
=L^{2,2}(\R^3)=L^2(\R^3),
\]
where $\theta\in(0,1)$ is defined by $\frac12 = \frac{1-\theta}{q} + \frac{\theta}{r_2}$.
Consequently,
\[
\|e^{-tL_1}\mathbb P\operatorname{div}\BF\|_2 \le C \|e^{-tL_1}\mathbb P\operatorname{div}\BF\|_{q,\infty}^{1-\theta}
\|e^{-tL_1}\mathbb P\operatorname{div}\BF\|_{r_2,\infty}^{\theta}.
\]
Using the two estimates above and the elementary embedding $\|\cdot\|_{q,\infty}\le C\|\cdot\|_q$, we obtain
\[
\|e^{-tL_1}\mathbb P\operatorname{div}\BF\|_2 
\le  C t^{-\frac32\bigl(\frac1q-\frac12\bigr)-\frac12}\|\BF\|_q.
\]
The proof for $e^{-tL_2}$ is analogous and is omitted. 
Let $\bar{\bar\delta}(q)$ denote the corresponding constant for $L_2$ and set 
$\delta_2(q)=\min\{\bar\delta(q),\bar{\bar\delta}(q)\}$. 
This completes the proof.
\end{enumerate}
\end{proof}

\begin{remark}[Dependence of the constants]
In \cref{prop:decay} and \cref{lemma:a}, the constant $C$ is independent 
of the time variable $t$ and of the specific Landau solution $\bU^{\bb}$, provided that $|\bb|$ remains within the smallness bound $\delta$. Its value depends only on the space dimension $n=3$, the Lebesgue exponents $q,r$ appearing in the respective estimates, and this smallness threshold $\delta$. 
\end{remark}

\section{Asymptotic stability of weak solutions}
\label{sec:proofs}

In this section, we prove \cref{thm:stability} and \cref{thm:decay}. 
 Given a weak solution $(\bw,\bB)$ to \cref{eqn:PerturbedEqn}, a duality argument yields, for all $s \ge 0$,
\begin{align*}
\|\bw(s)\|_2 &\leq \|e^{-sL_1} \bw_0\|_2 + \int_0^s \bigl\|e^{-(s-\tau)L_1} \mathbb{P} \nabla\cdot (\bw \otimes \bw - \bB \otimes \bB)\bigr\|_2 \, \mathrm{d}\tau,
\\
\|\bB(s)\|_2 &\leq \|e^{-sL_2} \bB_0\|_2 + \int_0^s \bigl\|e^{-(s-\tau)L_2} \mathbb{P} \nabla\cdot (\bw \otimes \bB - \bB \otimes \bw)\bigr\|_2 \, \mathrm{d}\tau.
\end{align*}
Starting from this consideration, the proofs of the main theorems rely on the strong energy inequality and on the following auxiliary proposition. Its first part provides an estimate for the linear terms, while the second part handles the nonlinear terms.

\begin{proposition} \label{lemma:a}
Let $L_i$, with $i =1,\,2$, be the linear operators defined in \cref{eqn:linearoperator}. Then the following statements hold.
\begin{propenum}
\item \label{it:lma1} There exists a constant $\delta_3>0$ such that if $0<|\bb|\le \delta_3$, then for every $\Bf\in L^2_{\sigma}(\R^3)$,
\begin{align}\label{eqn:linearpart}
    \lim\limits_{t\to\infty} \frac{2}{t} \int_{\frac{t}{2}}^t \|e^{-sL_i}\Bf\|_2 \, \mathrm d s=0.
\end{align}    
 \item \label{it:lma2} Let $\frac{6}{5}<r\le2$ and let $\delta_2 = \delta_2(r) > 0$ be the constant from \cref{it:lmdecay2}. If $0<|\bb|\le \delta_2$, then for all $\bu,\bv\in H^1(\R^3)$,
    \begin{align}\label{eqn:nonlinearpart}
    \|e^{-sL_i} \mathbb{P} \nabla\cdot (\bu\otimes \bv)\|_2 \le C s^{-\frac{3}{2}\left(\frac{1}{r}-\frac{1}{2}\right)-\frac{1}{2}} 
    (\|\bu\|_2\|\bv\|_2)^{\frac{3}{2r}-\frac{1}{2}} (\|\nabla\bu\|_2\|\nabla \bv\|_2)^{\frac{3}{2}(1-\frac{1}{r})}.
    \end{align}
\end{propenum}
\end{proposition}

\begin{proof} The two claims are consequences of \cref{prop:decay}.
\begin{enumerate}[label={(\roman*)}]
    \item Let $\delta_3=\min\{\delta_1(\frac{3}{2}),\delta_1(2)\}$, where $\delta_1(r)$ is the constant defined in \cref{it:lmdecay1}. 
For all $\Bf \in L^2_\sigma(\R^3)$ and every $\varepsilon > 0$, we can choose $\Bf_\varepsilon \in C_{c,\sigma}^{\infty}(\R^3)$ such that $\|\Bf - \Bf_\varepsilon\|_2 \le \varepsilon$. Then 
\[
\|e^{-sL_i} \Bf\|_2 
\le \|e^{-sL_i} (\Bf - \Bf_\varepsilon)\|_2 + \|e^{-sL_i} \Bf_\varepsilon\|_2 
\le C \varepsilon + C s^{-\frac{1}{4}} \|\Bf_\varepsilon\|_{\frac32},
\]
where in the last inequality we used \cref{eqn:decay1}. Therefore we have
\[
\frac{2}{t} \int_{\frac{t}{2}}^{t} \|e^{-sL_i} \Bf\|_2 \, \mathrm d s 
\le C\varepsilon + Ct^{-\frac{1}{4}} \|\Bf_\varepsilon\|_{\frac{3}{2}}.
\]
From this, we deduce that
\[
\lim_{t \to \infty} \frac{2}{t} \int_{\frac{t}{2}}^{t} \|e^{-sL_i} \Bf\|_2 \, \mathrm ds \le C \varepsilon.
\]
Since $\varepsilon$ is arbitrary, it follows that $\lim\limits_{t \to \infty} \frac{2}{t} \int_{\frac{t}{2}}^{t} \|e^{-sL_i} \Bf\|_2 \, \mathrm ds = 0$.\\

\item  Using the decay property of $e^{-tL_i}$ in \cref{it:lmdecay2}, we have the estimate
\[
\|e^{-sL_i} \mathbb{P} \nabla\cdot (\bu \otimes \bv)\|_2 
\leq C s^{-\frac{3}{2}\left(\frac{1}{r}-\frac{1}{2}\right) - \frac{1}{2}} \|\bu \otimes \bv\|_r.
\]
H\"older's inequality and Gagliardo--Nirenberg's inequality imply 
\[
\|\bu \otimes \bv\|_r 
\le \|\bu\|_{2r} \| \bv\|_{2r}
\le (\|\bu\|_{2}\|\bv\|_{2})^{1-\theta} 
(\|\nabla \bu\|_{2}\|\nabla \bv\|_{2})^{\theta},
\]
where $\theta$ satisfies $\frac{1}{2r}=\frac{1-\theta}{2}+\theta(\frac{1}{2}-\frac{1}{3})$; in particular, $\theta\in(0,1)$ precisely when $r\in(\frac{6}{5},2]$. Consequently,
\[
\|\bu \otimes \bv\|_r 
\le  (\|\bu\|_{2}\|\bv\|_{2})^{\frac{3}{2r}-\frac{1}{2}} 
(\|\nabla \bu\|_{2}\|\nabla \bv\|_{2})^{\frac{3}{2}(1-\frac{1}{r})},
\]
which completes the proof.
\end{enumerate}
\end{proof}

We now turn to the proof of \cref{thm:stability}.  

\begin{proof}[Proof of \cref{thm:stability}] 

\uline{Step 1.} \emph{Duhamel-type representation.} 
As $(\bw,\bB)$ is a weak solution to \cref{eqn:PerturbedEqn}, we have
\begin{align}\label{eqn:weakformulation1}
\langle \bw(s), \boldsymbol{\varphi}(s)  \rangle
+\int_0^s \Big[ 
\langle \bw\cdot  \nabla \bw, \boldsymbol{\varphi} \rangle
- \langle \bB\cdot  \nabla \bB, \boldsymbol{\varphi} \rangle
+ a_1(\bw,\boldsymbol{\vp})
\Big]\, \mathrm d\tau 
= \langle \bw(0), \boldsymbol{\varphi}(0)  \rangle 
+ \int_0^s \langle \bw, \boldsymbol{\varphi}_\tau \rangle \, \mathrm d\tau
\end{align}
for all $s\geq 0$ and all $\boldsymbol{\varphi}$ satisfying 
\[
\boldsymbol{\varphi} \in C\big( [0, \infty); L^2_\sigma(\R^3) \cap L^3(\R^3) \big), \quad 
\nabla \boldsymbol{\varphi} \in L^2_{\mathrm{loc}}\big( [0,\infty); L^2(\R^3) \big), \quad 
\partial_\tau \boldsymbol{\varphi} \in L^2_{\mathrm{loc}}\big( [0,\infty); (H^1_\sigma(\R^3))^* \big).
\]


There exists a constant $0<\delta_4<\min\{\ve_0, \frac{1}{4C_0}\}$ such that for any $0<|\bb|<\delta_4$ and any $\boldsymbol{\psi}\in C^{\infty}_{c,\sigma}(\R^3)$, the function $\boldsymbol{\vp}(\tau)=e^{-(s-\tau)L_i^*}\boldsymbol{\psi}$ belongs to the admissible test function space for $i=1,2$.
Indeed, for such $|\bb|$, the analyticity of the semigroup $e^{-tL_i^*}$ on $L^3_\sigma(\R^3)$ (see \cite[Theorem 3.1]{LZZ23}) ensures that $\boldsymbol{\vp}\in C([0,s];L^3_\sigma(\R^3))$. The remaining regularity requirements are then satisfied by \cref{prop:reg}.

Taking this $\boldsymbol{\vp}$ as a test function in \eqref{eqn:weakformulation1} and using $\langle\bw,\partial_\tau\boldsymbol{\vp}\rangle=\langle\bw,L_1^*\boldsymbol{\vp}\rangle=a_1(\bw,\boldsymbol{\vp})$, we obtain
\begin{align*}
\langle \bw(s), \boldsymbol{\psi} \rangle
= \Big\langle e^{-sL_1}\bw_0
- \int_0^s e^{-(s-\tau)L_1}\mathbb{P}(\bw\cdot\nabla\bw-\bB\cdot\nabla\bB)\, \mathrm d\tau, \boldsymbol{\psi} \Big\rangle.
\end{align*}
By duality, this gives
\begin{align}\label{eqn:estimate2-1}
\|\bw(s)\|_2 \leq \|e^{-sL_1} \bw_0\|_2 + \int_0^s \bigl\|e^{-(s-\tau)L_1} \mathbb{P} \nabla\cdot (\bw \otimes \bw - \bB \otimes \bB)\bigr\|_2 \, \mathrm{d}\tau.
\end{align}
Similarly, taking $\boldsymbol{\vp}(\tau)=e^{-(s-\tau)L_2^*}\boldsymbol{\psi}$ as a test function in \cref{eqn:weak}$_2$ yields
\begin{align}\label{eqn:estimate2-2}
\|\bB(s)\|_2 \leq \|e^{-sL_2} \bB_0\|_2 + \int_0^s \bigl\|e^{-(s-\tau)L_2} \mathbb{P} \nabla\cdot (\bw \otimes \bB - \bB \otimes \bw)\bigr\|_2 \, \mathrm{d}\tau.
\end{align}

\uline{Step 2.} \emph{Energy monotonicity and time averaging}. If $0<|\bb|<\min\{\ve_0,\frac{1}{4C_0}\}$, Hardy's inequality \cref{eqn:Hardy} and the strong energy inequality \cref{eqn:energyineq1} imply
\begin{align}\label{eqn:energyineq2}
\|\bw(t)\|_2^2 + \|\bB(t)\|_2^2 + 2(1-4C_0|\bb|) \int_s^t \left(\|\nabla \bw\|_2^2 + \|\nabla \bB\|_2^2\right) \, \mathrm d\tau \leq \|\bw(s)\|_2^2 + \|\bB(s)\|_2^2,
\end{align}
which yields
\begin{align}\label{eq:s-est-0}
\|\bw(t)\|_2^2 + \|\bB(t)\|_2^2 \leq \|\bw(s)\|_2^2 + \|\bB(s)\|_2^2
\end{align}
for almost every $s\geq 0$ (including $s=0$) and for all $t\geq s$.
Therefore, we obtain
\begin{align*}
\frac{\|\bw(t)\|_2 + \|\bB(t)\|_2}{\sqrt{2}} 
\leq \sqrt{\|\bw(t)\|_2^2 + \|\bB(t)\|_2^2} 
\leq \sqrt{\|\bw(s)\|_2^2 + \|\bB(s)\|_2^2} 
\leq \|\bw(s)\|_2 + \|\bB(s)\|_2.
\end{align*}
Integrating the inequality over $s\in[\frac{t}{2}, t]$ gives
\begin{align}\label{eqn:estimate1}
\frac{\|\bw(t)\|_2 + \|\bB(t)\|_2}{\sqrt{2}} 
\leq \frac{2}{t} \int_{\frac{t}{2}}^{t} \|\bw(s)\|_2 \, \mathrm ds + \frac{2}{t} \int_{\frac{t}{2}}^{t} \|\bB(s)\|_2 \, \mathrm ds. 
\end{align}

\uline{Step 3.} \emph{Combination of linear and nonlinear contributions.} Combining \cref{eqn:estimate2-1}, \cref{eqn:estimate2-2} and \cref{eqn:estimate1}, we deduce 
\begin{equation}\label{eqn:estimate3}
\begin{aligned}
\frac{\|\bw(t)\|_2 + \|\bB(t)\|_2}{\sqrt{2}} 
 &\leq \frac{2}{t} \int_{\frac{t}{2}}^t \bigl( \|e^{-sL_1} \bw_{0}\|_2 + \|e^{-sL_2} \bB_0\|_2 \bigr) \, \mathrm{d} s\\
& +  \frac{2}{t} \int_{\frac{t}{2}}^{t} \Bigg[ \int_0^s \Bigl( \|e^{-(s-\tau)L_1} \mathbb{P}\nabla\cdot (\bw \otimes \bw - \bB \otimes \bB)\|_2 
\\ & \qquad \qquad \qquad  + \|e^{-(s-\tau)L_2} \mathbb{P}\nabla\cdot (\bw\otimes \bB - \bB \otimes \bw)\|_2 \Bigr) \, \mathrm{d} \tau \Bigg] \mathrm{d}s.
\end{aligned}
\end{equation}

\uline{Step 4.} \emph{Linear contribution.} If $|\bb|\le \delta_3$, where $\delta_3$ is the constant in \cref{it:lma1} ,
\[
\lim\limits_{t\to \infty} \frac{2}{t} \int_{\frac{t}{2}}^t \bigl( \|e^{-sL_1} \bw_{0}\|_2 + \|e^{-sL_2} \bB_0\|_2 \bigr) \, \mathrm{d} s=0.
\]

\uline{Step 5.} \emph{Nonlinear contribution.} Inequality \cref{eqn:nonlinearpart} implies that
\begin{align*}
&\frac{2}{t} \int_{\frac{t}{2}}^{t} \left[ \int_0^s  \|e^{-(s-\tau)L_1} \mathbb{P}\nabla\cdot (\bw \otimes \bw)\|_2 \, \mathrm{d} \tau \right] \mathrm{d}s\\
& \le  \frac{C}{t} \int_{\frac{t}{2}}^{t} 
\left[ \int_0^s (s-\tau)^{-\frac{3}{2}(\frac{1}{r_1}-\frac{1}{2})-\frac{1}{2}} \|\bw(\tau)\|_2^{\frac{3}{r_1}-1}  h^{\frac
32(1-\frac{1}{r_1})} \, \mathrm{d} \tau \right] \mathrm{d}s,
\end{align*}
where $h(\tau)=\|\nabla \bw(\tau)\|_2^2$ and $r_1\in (\frac{6}{5},2]$ is a constant to be determined. 
\cref{eqn:energyineq2} implies $h\in L^{1}(0,\infty)$; moreover, by \eqref{eq:s-est-0}, $\|\bw(\tau)\|_2 \le \|\bw_0\|_2+\|\bB_0\|_2$.
Therefore
\[
\frac{C}{t} \int_{\frac{t}{2}}^{t} 
\left[ \int_0^s (s-\tau)^{-\frac{3}{2}(\frac{1}{r_1}-\frac{1}{2})-\frac{1}{2}} \|\bw(\tau)\|_2^{\frac{3}{r_1}-1}  h^{\frac
32(1-\frac{1}{r_1})} \, \mathrm{d} \tau \right] \mathrm{d}s
\le 
\frac{C}{t} \int_{\frac{t}{2}}^{t} 
|\cdot|^{-\frac{3}{2}(\frac{1}{r_1}-\frac{1}{2})-\frac{1}{2}} \ast h^{\frac{3}{2}(1-\frac{1}{r_1})}\, \mathrm{d}s.
\]
Using \cref{eqn:weakLpnorm}, we have
\begin{align*}
\frac{C}{t} \int_{\frac{t}{2}}^{t} 
|\cdot|^{-\frac{3}{2}(\frac{1}{r}-\frac{1}{2})-\frac{1}{2}} \ast h^{\frac{3}{2}(1-\frac{1}{r_1})}\, \mathrm{d}s
& \le Ct^{-\frac{1}{\alpha_1}}\left\| |\cdot|^{-\frac{3}{2}(\frac{1}{r_1}-\frac{1}{2})-\frac{1}{2}} \ast h^{\frac{3}{2}(1-\frac{1}{r_1})} \right\|_{\alpha_1,\infty}\\
& \le  Ct^{-\frac{1}{\alpha_1}}
\left\| |\cdot|^{-\frac{3}{2}(\frac{1}{r_1}-\frac{1}{2})-\frac{1}{2}}\right\|_{c_1,\infty} 
\left\| h^{\frac{3}{2}(1-\frac{1}{r_1})} \right\|_{c_2}.
\end{align*}
Here we used \cref{eqn:Youngineq} and the embedding $L^p \subset L^{p,\infty}$ in the second line, and $\alpha_1,c_1,c_2>1$ are constants satisfying
\[
\frac{1}{c_1}=\frac{3}{2}\left(\frac{1}{r_1}-\frac{1}{2}\right)+\frac{1}{2},\quad
\frac{1}{c_2}=\frac{3}{2}\left(1-\frac{1}{r_1}\right),
\]
and
\[
1+\frac{1}{\alpha_1}=\frac{1}{c_1}+\frac{1}{c_2}.
\]
Note that $\frac{1}{c_1}+\frac{1}{c_2}=\frac{5}{4}$, so $\frac{1}{\alpha_1}=\frac{1}{4}$ provided $c_1,c_2\in (1,\infty)$. For instance, taking $r_1=\frac{3}{2}$ gives $c_1,c_2\in(1,\infty)$.
Therefore, if $|\bb|\le \delta_2(\frac{3}{2})$ with $\delta_2(\frac{3}{2})$ as in \cref{it:lma2}, 
\[
\frac{2}{t} \int_{\frac{t}{2}}^{t} \left[ \int_0^s \|e^{-(s-\tau)L_1} \mathbb{P}\nabla\cdot (\bw \otimes \bw)\|_2 \, \mathrm{d} \tau \right] \mathrm{d}s \le C t^{-\frac{1}{4}}.
\]
Similarly, 
\begin{align*}
&\frac{2}{t} \int_{\frac{t}{2}}^{t} \left[ \int_0^s \Bigl( \|e^{-(s-\tau)L_1} \mathbb{P}\nabla\cdot (\bw \otimes \bw - \bB \otimes \bB)\|_2 
+ \|e^{-(s-\tau)L_2} \mathbb{P}\nabla\cdot (\bw\otimes \bB - \bB \otimes \bw)\|_2 \Bigr) \, \mathrm{d} \tau \right] \, \mathrm{d}s \\
&\le Ct^{-\frac{1}{4}}
\end{align*}
provided $|\bb| \le \delta_2(\frac{3}{2})$.

\uline{Step 6.} \emph{Conclusion.} In conclusion, if $0<|\bb|\le \delta$, where $\delta=\min\{ \delta_4, \delta_3, \delta_2(\frac{3}{2})\}$, \cref{eqn:estimate3} yields
\[
\lim\limits_{t\to \infty}\|\bw(t)\|_2 
=\lim\limits_{t\to \infty} \|\bB(t)\|_2=0.
\]
\end{proof}

We now establish the decay estimate for the perturbed equation \cref{eqn:PerturbedEqn}. The proof proceeds by iteration.

\begin{proof}[Proof of \cref{thm:decay}]

\uline{Step 1.} \emph{Linear decay estimate.} By
\cref{it:lmdecay1}
\[
\frac{2}{t} \int_0^t \bigl( \|e^{-sL_1} \bw_{0}\|_2 + \|e^{-sL_2} \bB_0\|_2 \bigr) \, \mathrm{d} s 
\le Ct^{-\frac{3}{2}(\frac{1}{q}-\frac{1}{2})} (\|\bw_0\|_q + \|\bB_0\|_q).
\]

\uline{Step 2.} \emph{Baseline decay bound.} 
From the estimates obtained in the proof of \cref{thm:stability}, 
\begin{align*}
&\frac{2}{t} \int_{\frac{t}{2}}^{t} \left[ \int_0^s \Bigl( \|e^{-(s-\tau)L_1} \mathbb{P}\nabla\cdot (\bw \otimes \bw - \bB \otimes \bB)\|_2 
+ \|e^{-(s-\tau)L_2} \mathbb{P}\nabla\cdot (\bw\otimes \bB - \bB \otimes \bw)\|_2 \Bigr) \, \mathrm{d} \tau \right] \, \mathrm{d}s \\
&\le Ct^{-\frac{1}{4}}.
\end{align*}
Combining these estimates with \cref{eqn:estimate3}, we have
\[
\|\bw(t)\|_2+\|\bB(t)\|_2 \le Ct^{-\frac{3}{2}(\frac{1}{q}-\frac{1}{2})} + Ct^{-\frac{1}{4}}.
\]

\uline{Step 3.} \emph{Decay regimes according to the integrability exponent $q$.} Thus \cref{thm:decay} holds whenever $-\frac{1}{4} \le -\frac{3}{2}(\frac{1}{q}-\frac{1}{2})$, that is, $\frac{3}{2} \le q <2$. If $1 \le q < \frac{3}{2}$, $\|\bw(t)\|_2+\|\bB(t)\|_2 \le Ct^{-\frac{1}{4}}$.

\uline{Step 4.} \emph{Bootstrap improvement of decay for $1 \le q < \frac{3}{2}$.}
Similarly, \cref{eqn:nonlinearpart} shows
\begin{align*}
&\frac{2}{t} \int_{\frac{t}{2}}^{t} \left[ \int_0^s  \|e^{-(s-\tau)L_1} \mathbb{P}\nabla\cdot (\bw \otimes \bw)\|_2 \, \mathrm{d} \tau \right] \mathrm{d}s\\
& \le  \frac{C}{t} \int_{\frac{t}{2}}^{t} 
\left[ \int_0^s (s-\tau)^{-\frac{3}{2}(\frac{1}{r_2}-\frac{1}{2})-\frac{1}{2}} \|\bw(\tau)\|_2^{\frac{3}{r_2}-1}  h^{\frac32(1-\frac{1}{r_2})} \, \mathrm{d} \tau \right] \mathrm{d}s.
\end{align*}
Here $h(\tau)=\|\nabla\bw(\tau)\|_2^2$, and $r_2\in (\frac{6}{5},2]$ is a constant to be determined. Using $\|\bw(\tau)\|_2 \le C\tau^{-\frac{1}{4}}$, we obtain
\begin{align*}
&\frac{C}{t} \int_{\frac{t}{2}}^{t} 
\left[ \int_0^s (s-\tau)^{-\frac{3}{2}(\frac{1}{r_2}-\frac{1}{2})-\frac{1}{2}} \|\bw(\tau)\|_2^{\frac{3}{r_2}-1}  h^{\frac32(1-\frac{1}{r_2})} \, \mathrm{d} \tau \right] \mathrm{d}s\\
& \le   \frac{C}{t} \int_{\frac{t}{2}}^{t} 
\left[ \int_0^s (s-\tau)^{-\frac{3}{2}(\frac{1}{r_2}-\frac{1}{2})-\frac{1}{2}}  \tau^{-\frac{1}{4}(\frac{3}{r_2}-1)}  h^{\frac{3}{2}(1-\frac{1}{r_2})} \, \mathrm{d} \tau \right] \mathrm{d}s\\
&=\frac{C}{t} \int_{\frac{t}{2}}^{t} 
|\cdot|^{-\frac{3}{2}(\frac{1}{r_2}-\frac{1}{2})-\frac{1}{2}} \ast ( |\cdot|^{-\frac{1}{4}(\frac{3}{r_2}-1)} \, h^{\frac{3}{2}(1-\frac{1}{r_2})})\, \mathrm{d}s,
\end{align*}
Recalling
\cref{eqn:weakLpnorm}, \cref{eqn:Youngineq}, and \cref{eqn:Holderineq}, we deduce 
\begin{align*}
&\frac{C}{t} \int_{\frac{t}{2}}^{t} 
|\cdot|^{-\frac{3}{2}(\frac{1}{r_2}-\frac{1}{2})-\frac{1}{2}} \ast ( |\cdot|^{-\frac{1}{4}(\frac{3}{r_2}-1)} \, h^{\frac{3}{2}(1-\frac{1}{r_2})})\, \mathrm{d}s \\
& \le 
C t^{-\frac{1}{\alpha_2}}
\left\| |\cdot|^{-\frac{3}{2}(\frac{1}{r_2}-\frac{1}{2})-\frac{1}{2}} \ast ( |\cdot|^{-\frac{1}{4}(\frac{3}{r_2}-1)} \, h^{\frac{3}{2}(1-\frac{1}{r_2})}) \right\|_{\alpha_2,\infty}\\
& \le  C t^{-\frac{1}{\alpha_2}}
\left\| |\cdot|^{-\frac{3}{2}(\frac{1}{r_2}-\frac{1}{2})-\frac{1}{2}} \right\|_{c_3,\infty}
\left\| |\cdot|^{-\frac{1}{4}(\frac{3}{r_2}-1)} \right\|_{c_4,\infty}
\left\| h^{\frac{3}{2}(1-\frac{1}{r_2})} \right\|_{c_5,\infty},
\end{align*}
where $\alpha_2,c_3,c_4,c_5>1$ are constants satisfying
\[
\frac{1}{c_3}=\frac{3}{2}\left(\frac{1}{r_2}-\frac{1}{2}\right)+\frac{1}{2},\quad
\frac{1}{c_4}= \frac{1}{4}\left(\frac{3}{r_2}-1\right),\quad 
\frac{1}{c_5}=\frac{3}{2}\left(1-\frac{1}{r_2}\right),
\]
and
\[
0< \frac{1}{c_4}+\frac{1}{c_5} <1,\quad 1+\frac{1}{\alpha_2}=\frac{1}{c_3}+\frac{1}{c_4}+\frac{1}{c_5}.
\]
Taking $r_2=\frac{15}{8}$ gives $\frac{1}{\alpha_2}=\frac{2}{5}$, which yields
\[
\|\bw(t)\|_2+\|\bB(t)\|_2 \le Ct^{-\frac{3}{2}(\frac{1}{q}-\frac{1}{2})} + Ct^{-\frac{2}{5}}.
\]
Thus \cref{thm:decay} holds provided $-\frac{2}{5} \le -\frac{3}{2}(\frac{1}{q}-\frac{1}{2})$, i.\,e., $\frac{30}{23} \le q <\frac{3}{2}$.

\uline{Step 
5.} \emph{Final bootstrap and completion of the proof.} If $1 \le q < \frac{30}{23}$, $\|\bw(t)\|_2+\|\bB(t)\|_2 \le Ct^{-\frac{2}{5}}$. Analogously,
\begin{align*}
&\frac{2}{t} \int_{\frac{t}{2}}^{t} \left[ \int_0^s \|e^{-(s-\tau)L_1} \mathbb{P}\nabla\cdot (\bw \otimes \bw)\|_2 \, \mathrm{d} \tau \right] \mathrm{d}s\\
& \le  C t^{-\frac{1}{\alpha_3}}
\left\| |\cdot|^{-\frac{3}{2}(\frac{1}{r_3}-\frac{1}{2})-\frac{1}{2}} \right\|_{c_6,\infty}
\left\| |\cdot|^{-\frac{2}{5}(\frac{3}{r_3}-1)} \right\|_{c_7,\infty}
\left\| h^{\frac{3}{2}(1-\frac{1}{r_3})} \right\|_{c_8,\infty}.
\end{align*}
Here $r_3\in (\frac{6}{5},2]$ and $\alpha_3,c_6,c_7,c_8>1$ are constants satisfying
\[
\frac{1}{c_6}=\frac{3}{2}\left(\frac{1}{r_3}-\frac{1}{2}\right)+\frac{1}{2},\quad
\frac{1}{c_7}= \frac{2}{5}\left(\frac{3}{r_3}-1\right),\quad 
\frac{1}{c_8}=\frac{3}{2}\left(1-\frac{1}{r_3}\right),
\]
and
\[
0< \frac{1}{c_7}+\frac{1}{c_8} <1, \quad 1+\frac{1}{\alpha_3}=\frac{1}{c_6}+\frac{1}{c_7}+\frac{1}{c_8}.
\]
Choosing $r_3=\frac{4}{3}$ gives $\frac{1}{\alpha_3}=\frac{3}{4}$, which yields
\[
\|\bw(t)\|_2+\|\bB(t)\|_2 \le Ct^{-\frac{3}{2}\left(\frac{1}{q}-\frac{1}{2}\right)} + Ct^{-\frac{3}{4}}.
\]
Combining this estimate with \eqref{eq:s-est-0} at $s=0$ completes the proof.
\end{proof}

\begin{remark}[Refinement of the decay rate]
In the proof of \cref{thm:decay}, the constants $r_2$ and $r_3$ can be chosen so that $r_2 > 3/2$ is arbitrarily close to $3/2$ and $r_3 > \frac{6}{5}$ is arbitrarily close to $\frac{6}{5}$. With such a choice, we have $\frac{1}{\alpha_3} = 1 - \varepsilon$ for an arbitrarily small $\varepsilon > 0$, which gives the decay estimate
\[
\|\bw(t)\|_2+\|\bB(t)\|_2 \le C t^{-\frac{3}{2}\left(\frac{1}{q}-\frac{1}{2}\right)} + C t^{-(1-\varepsilon)}.
\]
\end{remark}

\vspace{0.5cm}

\section*{Acknowledgments}

N.~De Nitti is a member of the Gruppo Nazionale per l'Analisi Matematica, la Probabilità e le loro Applicazioni (GNAMPA) of the Istituto Nazionale di Alta Matematica (INdAM) and has received support from the INdAM--GNAMPA Project 2026 \textit{Modelli Non-locali in Fluidodinamica, Traffico ed Elasticità} (CUP:~E53C25002010001). He is grateful to M.~Sammartino and V.~Sciacca for helpful discussions on MHD models at the University of Palermo. He acknowledges Soochow University, Duke Kunshan University, and Shanghai Jiao Tong University for their kind hospitality.

Y.~Wang is partially supported by NSFC grant 12671279, the Natural Science Foundation of Jiangsu Province (grant~BK20240147), Jiangsu Provincial Scientific Research Center of Applied Mathematics(No. BK20233002).

S.~Zhang is partially supported by the Postgraduate Research \&
Practice Innovation Program of Jiangsu Province (grant~KYCX24\_3285). He acknowledges the kind hospitality of Duke Kunshan University and Shanghai Jiao Tong University.

\appendix

\section{Construction of weak solutions}
\label{sec:appendix}

In this appendix, we prove the existence of weak solutions to \cref{eqn:PerturbedEqn} that satisfy the strong energy inequality \cref{eqn:energyineq1}. 
The existence of weak solutions is proved via the retarded mollification method (see \cite[Appendix]{CKN82} and \cite[Theorem 3.5, Theorem 3.9]{Tsai2018}). The global convergence that yields the strong energy inequality follows the argument in \cite[Proposition 14.1]{LR02}.
We adopt the notations from \cref{ssec:perturb-setting} and \cref{ssec:main}.

We first prove the existence and uniqueness of solutions to the linearized MHD system by using the Faedo--Galerkin method, following the arguments in \cite[Theorem 3.5]{Tsai2018} and \cite[Chapter III, Theorem 1.1]{Temam1979}.

\begin{lemma}[Existence and uniqueness for the linearized MHD system]\label{apdx-mhd-lm}
Let $\bA,\bD\in C^1(\mathbb R^3\times[0,\infty))\cap L^\infty(\mathbb R^3\times[0,\infty))$
with $\nabla\cdot\bA=\nabla\cdot\bD=0$.
If $|\bb|<\min\{\ve_0, \frac{1}{4C_0}\}$, where $\ve_0$ and $C_0$ are as in \cref{ssec:existence} (see also \eqref{eqn:LandauEst}),  then for any $\bw_0,\bB_0\in L^2_\sigma(\mathbb{R}^3)$ and any $T>0$, there exists a unique weak solution $\bw,\,\bB\in C([0,T];L^2_\sigma(\mathbb{R}^3))\cap L^2(0,T;H^1_\sigma(\mathbb{R}^3))$ to the linearized problem
\begin{equation}\label{apdx-lm-LinEqn}
\begin{cases}
\p_t\bw +\bA\cdot\nabla\bw+\bU^{\bb}\cdot\nabla\bw+\bw\cdot\nabla\bU^{\bb}
+\nabla\pi =\Delta\bw+\bD\cdot\nabla\bB,\quad &(\bx,t)\in\mathbb{R}^3\times(0,T),\\
\partial_t\bB+\bA\cdot\nabla\bB+\bU^{\bb}\cdot\nabla\bB + \nabla \Pi
=\Delta\bB+\bD\cdot\nabla\bw+\bB\cdot\nabla\bU^{\bb},\quad &(\bx,t)\in\mathbb{R}^3\times(0,T),\\
\nabla\cdot\bw=0,\quad \nabla\cdot\bB=0,\quad &(\bx,t)\in\mathbb{R}^3\times(0,T),\\
\bw(\bx, 0)=\bw_0(\bx),\quad \bB(\bx,0)=\bB_0(\bx),\quad &\bx\in\mathbb{R}^3.
\end{cases}
\end{equation}
Moreover, $\partial_t\bw, \partial_t\bB\in L^2\bigl(0,T;(H^1_\sigma(\mathbb{R}^3))^*\bigr)$, and the energy identity
\begin{align}\label{apdx-lm-EnergyId}
\begin{aligned}
&\|\bw(t)\|_2^2+\|\bB(t)\|_2^2
+2\int_0^t\Bigl(\|\nabla\bw(s)\|_2^2+\|\nabla\bB(s)\|_2^2\Bigr)\,\mathrm{d}s\\
&\quad
+2\int_0^t\Bigl(\langle(\bw(s)\cdot\nabla)\bU^{\bb},\bw(s)\rangle
-\langle(\bB(s)\cdot\nabla)\bU^{\bb},\bB(s)\rangle\Bigr)\,\mathrm{d}s
=\|\bw_0\|_2^2+\|\bB_0\|_2^2
\end{aligned}
\end{align}
holds for all $t\in[0,T]$.
\end{lemma}

Here we note that the gradient term $\nabla\Pi$ in the magnetic equation of the linearized MHD system \eqref{apdx-lm-LinEqn}  serves only as a Lagrange multiplier to enforce $\nabla\cdot B=0$. It is an auxiliary term rather than a physical magnetic pressure.

\begin{proof}
\underline{Step 1.}
\emph{Galerkin method setup.}
Choose a family $\{\bg_j\}_{j\ge1}\subset C^\infty_{c,\sigma}(\mathbb{R}^3)$
whose linear span is dense in $H^1_\sigma(\mathbb{R}^3)$, orthonormal with
respect to the $L^2_\sigma$ inner product. We look for approximate solutions
$\bw_m(t)=\sum_{i=1}^m d_{im}(t)\bg_i,\,
\bB_m(t)=\sum_{i=1}^m e_{im}(t)\bg_i$,
satisfying, for $j=1,\cdots,m$,
\begin{align}\label{apdx-lm-ODE}
\begin{aligned}
&\langle\partial_t\bw_m,\bg_j\rangle
+\langle\nabla\bw_m,\nabla\bg_j\rangle
+\langle\bA\cdot\nabla\bw_m,\bg_j\rangle
+\langle\bU^{\bb}\cdot\nabla\bw_m,\bg_j\rangle
+\langle\bw_m\cdot\nabla\bU^{\bb},\bg_j\rangle\\
&\quad =\langle\bD\cdot\nabla\bB_m,\bg_j\rangle,\\
&\langle\partial_t\bB_m,\bg_j\rangle
+\langle\nabla\bB_m,\nabla\bg_j\rangle
+\langle\bA\cdot\nabla\bB_m,\bg_j\rangle
+\langle\bU^{\bb}\cdot\nabla\bB_m,\bg_j\rangle\\
&\quad =\langle\bD\cdot\nabla\bw_m,\bg_j\rangle
+\langle\bB_m\cdot\nabla\bU^{\bb},\bg_j\rangle,\\
&\bw_m(0)=\sum_{i=1}^m\langle\bw_0,\bg_i\rangle\bg_i,\qquad
\bB_m(0)=\sum_{i=1}^m\langle\bB_0,\bg_i\rangle\bg_i.
\end{aligned}
\end{align}
The pressure term vanishes after testing against divergence-free functions.
For $|\bb|<\ve_0$, the estimates from \cref{thm:analyticsemigroup} yield
\begin{equation}\label{apdx-mhd-lm-hardy}
\begin{aligned}
|\langle\bU^{\bb}\cdot\nabla \bu,\bv\rangle|
&\le 2C_0|\bb|\,\|\nabla \bu\|_2\|\nabla \bv\|_2,\\
|\langle(\bu\cdot\nabla)\bU^{\bb},\bv\rangle|
&\le 4C_0|\bb|\,\|\nabla \bu\|_2\|\nabla \bv\|_2,
\end{aligned}
\qquad
\bu,\bv\in H^1_\sigma(\mathbb{R}^3).
\end{equation}
Thus all terms involving $\bU^{\bb}$ in the ODE system are bounded. Since $\bA,\bD\in C^1$, the right-hand side of the ODE system is linear and continuous in the coefficient vector, so the ODE system has a unique global solution on $[0,T]$ by standard ODE theory.

\underline{Step 2.}
\emph{Energy estimates.}
Multiplying \eqref{apdx-lm-ODE}$_{1j}$ by $d_{jm}$, \eqref{apdx-lm-ODE}$_{2j}$ by
$e_{jm}$, summing over $j$, and adding the resulting identities, we obtain
\begin{align*}
\begin{aligned}
&\frac12\frac{\mathrm{d}}{\mathrm{d}t}\Bigl(\|\bw_m\|_2^2+\|\bB_m\|_2^2\Bigr)
+\|\nabla\bw_m\|_2^2+\|\nabla\bB_m\|_2^2\\
&\qquad
+\langle(\bw_m\cdot\nabla)\bU^{\bb},\bw_m\rangle
-\langle(\bB_m\cdot\nabla)\bU^{\bb},\bB_m\rangle
=0.
\end{aligned}
\end{align*}
Here we used
\[
\langle\bA\cdot\nabla\bw_m,\bw_m\rangle
=\langle\bA\cdot\nabla\bB_m,\bB_m\rangle
=\langle\bU^{\bb}\cdot\nabla\bw_m,\bw_m\rangle
=\langle\bU^{\bb}\cdot\nabla\bB_m,\bB_m\rangle
=0,
\]
because $\bA,\bU^{\bb}$ are divergence-free, and $\langle\bD\cdot\nabla\bB_m,\bw_m\rangle
+\langle\bD\cdot\nabla\bw_m,\bB_m\rangle
=0$, which follows by integration by parts and $\nabla\cdot\bD=0$.
Applying \eqref{apdx-mhd-lm-hardy}, we obtain
\[
\frac12\frac{\mathrm{d}}{\mathrm{d}t}\Bigl(\|\bw_m\|_2^2+\|\bB_m\|_2^2\Bigr)
+(1-4C_0|\bb|)\Bigl(\|\nabla\bw_m\|_2^2+\|\nabla\bB_m\|_2^2\Bigr)
\le 0.
\]
Integrating in time and using $4C_0|\bb|<1$, we obtain the uniform bound for $t \in [0,T]$,
\[
\|\bw_m(t)\|_2^2+\|\bB_m(t)\|_2^2
+2(1-4C_0|\bb|)\int_0^t\Bigl(\|\nabla\bw_m\|_2^2+\|\nabla\bB_m\|_2^2\Bigr)\,\mathrm{d}s
\le \|\bw_0\|_2^2+\|\bB_0\|_2^2.
\]
Hence $\{\bw_m\},\{\bB_m\}$ are uniformly bounded in $L^\infty(0,T;L^2_\sigma(\mathbb{R}^3))\cap L^2(0,T;H^1_\sigma(\mathbb{R}^3))$.

\underline{Step 3.} \emph{Weak convergence and passage to the limit.}
The uniform bounds from Step 2 imply the existence of a subsequence, and limits $\bw,\bB$ such that
\begin{align}
\begin{aligned}\label{apdx-lm-WeakConv}
\bw_m\rightharpoonup\bw,\quad \bB_m\rightharpoonup\bB \quad &\text{weakly in } L^2(0,T;H^1_\sigma(\mathbb R^3)), \\
\bw_m\overset{*}{\rightharpoonup}\bw,\quad \bB_m\overset{*}{\rightharpoonup}\bB \quad &\text{weakly-$*$ in } L^\infty(0,T;L^2_\sigma(\mathbb R^3)).
\end{aligned}
\end{align}
For any $\phi \in C^1_c([0,T))$, integrating $\phi(t)\cdot$ \eqref{apdx-lm-ODE}$_{1j}$ on $[0,T]$, and passing to the limit, we obtain
\begin{align}\label{apdx-lm-Identity1}
\begin{aligned}
&\int_0^T \Big[-\langle \bw,\p_t(\phi\bg_j)\rangle 
+ \langle \nabla\bw,\nabla(\phi\bg_j)\rangle 
+ \langle (\bA+\bU^{\bb})\cdot\nabla\bw,\phi\bg_j\rangle \Big] \mathrm{d}t \\
&\qquad +\int_0^T \Big[
\langle (\bw\cdot\nabla)\bU^{\bb},\phi\bg_j\rangle 
- \langle \bD\cdot\nabla\bB,\phi\bg_j\rangle \Big] \mathrm{d}t 
=\langle \bw_0,\phi(0)\bg_j\rangle .
\end{aligned}
\end{align}
The limit passage to \eqref{apdx-lm-Identity1} is justified as follows.
The term containing the time derivative is handled by the weak-$*$ convergence
in \eqref{apdx-lm-WeakConv}, since $\partial_t(\phi\bg_j)\in L^1(0,T;L^2_\sigma(\mathbb R^3))$.
All remaining terms are continuous linear functionals on
$L^2(0,T;H^1_\sigma(\mathbb R^3))$; in particular, the terms involving $\bU^{\bb}$ are continuous
by the Hardy-type estimates \eqref{apdx-mhd-lm-hardy}. Hence they pass to the
limit by the weak convergence in \eqref{apdx-lm-WeakConv}.
The same argument applied to \eqref{apdx-lm-ODE}$_{2j}$ yields the weak formulation
for $\bB$.

Since the identity \eqref{apdx-lm-Identity1} holds for every $j$, and finite linear combinations of $\{\bg_j\}$ are dense in $H^1_\sigma(\mathbb R^3)$, we conclude that $(\bw,\bB)$ satisfies the weak formulation of \eqref{apdx-lm-LinEqn}.

\underline{Step 4.}
\emph{Time regularity and initial condition.}
By Step 3, the identity \eqref{apdx-lm-Identity1} holds with $\bg_j$ replaced by any $\boldsymbol\chi\in H^1_\sigma(\mathbb R^3)$. Thus we get 
\begin{align}\label{apdx-lm-weakderivative}
\langle \p_t\bw, \boldsymbol{\chi} \rangle
= - \langle \nabla\bw, \nabla \boldsymbol{\chi}\rangle 
- \langle (\bA+\bU^{\bb})\cdot\nabla\bw, \boldsymbol{\chi}\rangle
- \langle (\bw\cdot\nabla)\bU^{\bb},\boldsymbol{\chi}\rangle 
+ \langle \bD\cdot\nabla\bB, \boldsymbol{\chi}\rangle
\end{align}
in the weak sense. 
By the boundedness of $\bA,\bD$ and \eqref{apdx-mhd-lm-hardy},
\[
|\langle\partial_t\bw,\boldsymbol\chi\rangle|
\le
C\bigl(\|\nabla\bw\|_2 + \|\nabla\bB\|_2 \bigr)
\|\boldsymbol\chi\|_{H^1},
\]
and similarly for $\partial_t\bB$. Here $C$ depends on $\|\bA\|_{L^\infty(\mathbb R^3\times[0,\infty))},
\|\bD\|_{L^\infty(\mathbb R^3\times[0,\infty))}, C_0$ and $|\bb|$, and is therefore independent of $t$.
Thus
$\partial_t\bw,\partial_t\bB\in L^2(0,T;(H^1_\sigma(\R^3))^*)$. Since
$\bw,\bB\in L^2(0,T;H^1_\sigma(\R^3))$, by
\cite[Lemma 3.3(iii)]{Tsai2018}, after redefinition on a null set,
$\bw,\bB\in C([0,T];L^2_\sigma(\R^3))$. Taking $\phi(0)=1$ in \eqref{apdx-lm-Identity1} and using the density of $\{\bg_j\}$ in $L^2_\sigma(\R^3)$, we get $\bw(0)=\bw_0$; similarly $\bB(0)=\bB_0$.

\underline{Step 5.}
\emph{Energy identity and uniqueness.}
Since $\bw\in L^2(0,T;H^1_\sigma(\R^3))$, testing \eqref{apdx-lm-weakderivative} with $\bw$ and using $\frac{\mathrm{d}}{\mathrm{d}t}\|\bw(t)\|_2^2=2\langle \p_t\bw(t), \bw(t) \rangle$ (see \cite[Lemma 3.3(iii)]{Tsai2018}), we obtain
\[
\frac12 \frac{\mathrm{d}}{\mathrm{d}t} \|\bw(t)\|_2^2 
+ \|\nabla\bw(t)\|_2^2  
+ \langle (\bw(t)\cdot\nabla)\bU^{\bb}, \bw(t)\rangle 
- \langle \bD(t)\cdot\nabla\bB(t), \bw(t)\rangle=0.
\]
Here the vanishing of $\langle(\bA+\bU^\bb)\cdot\nabla\bw,\bw\rangle$ follows by approximation, Hardy-type estimates \eqref{apdx-mhd-lm-hardy}, and $\nabla\cdot(\bA+\bU^{\bb})=0$.
Similarly, testing the weak form for $\bB$ with $\bB$ gives
\[
\frac12\frac{\mathrm{d}}{\mathrm{d}t}\|\bB(t)\|_2^2
+\|\nabla\bB(t)\|_2^2
-\langle\bD(t)\cdot\nabla\bw(t),\bB(t)\rangle
-\langle\bB(t)\cdot\nabla\bU^{\bb}, \bB(t)\rangle
=0.
\]
Adding the two identities and using $\nabla\cdot\bD=0$, we obtain
\[
\frac12\frac{\mathrm{d}}{\mathrm{d}t}\Bigl(\|\bw\|_2^2+\|\bB\|_2^2\Bigr)
+\|\nabla\bw\|_2^2+\|\nabla\bB\|_2^2
+\langle(\bw\cdot\nabla)\bU^{\bb},\bw\rangle
-\langle(\bB\cdot\nabla)\bU^{\bb},\bB\rangle=0
\]
in the weak sense. 
Integrating yields the energy identity.

For uniqueness, let $(\bw_1,\bB_1),(\bw_2,\bB_2)$ be two weak solutions with the
same data, and set $\widetilde{\bw}=\bw_1-\bw_2$, $\widetilde\bB=\bB_1-\bB_2$. Then $(\widetilde\bw,\widetilde\bB)$
solves the same linear system with zero initial data. Testing with $(\widetilde\bw,\widetilde\bB)$
and using \eqref{apdx-mhd-lm-hardy} gives
\[
\|\widetilde\bw(t)\|_2^2+\|\widetilde\bB(t)\|_2^2
+2(1-4C_0|\bb|)\int_0^t \Bigl(\|\nabla\widetilde\bw(s)\|_2^2+\|\nabla\widetilde\bB(s)\|_2^2\Bigr)\, \mathrm{d}s
\le0.
\]
The uniqueness follows from $4C_0|\bb|<1$.
\end{proof}

\begin{theorem}[Existence of weak solutions with strong energy inequality]\label{thm:MHD-existence}
Assume $0<|\bb|<\min\{\ve_0, \frac{1}{4C_0}\}$, where $\ve_0$ and $C_0$ are as in \cref{ssec:existence} (see also \eqref{eqn:LandauEst}). For any initial data
$\bw_0,\bB_0\in L^2_\sigma(\R^3)$, there exists a weak solution
$(\bw,\bB,\pi)$ to the perturbed MHD system \eqref{eqn:PerturbedEqn} satisfying the strong energy inequality
\[
\begin{aligned}
&\|\bw(t)\|_2^2+\|\bB(t)\|_2^2
+2\int_s^t\Bigl(\|\nabla\bw(\tau)\|_2^2+\|\nabla\bB(\tau)\|_2^2\Bigr)\,\mathrm{d}\tau\\
& \le
\|\bw(s)\|_2^2+\|\bB(s)\|_2^2
-2\int_s^t\Bigl(\langle(\bw(\tau)\cdot\nabla)\bU^{\bb},\bw(\tau)\rangle
-\langle(\bB(\tau)\cdot\nabla)\bU^{\bb},\bB(\tau)\rangle\Bigr)\,\mathrm{d}\tau
\end{aligned}
\]
for almost every $s\ge0$ (including $s=0$) and all $t\ge s$.
\end{theorem}

\begin{proof}
The proof parallels the Navier--Stokes case. The existence of weak solutions is established as in \cite[Appendix]{CKN82} and \cite[Theorem 3.9]{Tsai2018}, while the global strong convergence follows from \cite[Proposition 14.1]{LR02}.

\underline{Step 1.} \emph{Construction of approximate solutions.}
Fix a nonnegative function $\psi\in C_c^\infty(\R^3\times\R)$ supported in $B_1(0)\times(1,2)$ and normalized so that $\iint \psi\,\mathrm{d}\bx\,\mathrm{d}t=1$.
For $\delta>0$, define $\psi_\delta(\bx,t)=\delta^{-4}\psi(\bx/\delta,t/\delta)$. 

Given $m\in \mathbb{Z}_+$, set $\delta=1/m$, and define  the retarded mollification
\[
\Psi_m(f)(\bx,t)=\int_{-\infty}^{\infty}\int_{\R^3}
\psi_\delta(\bx-\boldsymbol{y},t-s)\,f(\boldsymbol{y},s)\,\mathrm{d}\boldsymbol{y}\,\mathrm{d}s.
\]
We construct the approximate solutions $\bw_m,\bB_m,$ inductively on intervals $I_k=[k\delta,(k+1)\delta]$, $k=0,1,\cdots$. For $t\le0$, we extend $\bw_m(t)\equiv0$ and $\bB_m(t)\equiv0$ (see also \cite[eq.~(A.34)]{CKN82}). On each $I_k$, set $\bA_m(\bx,t):=\Psi_m(\bw_m)(\bx,t)$ and $\bD_m(\bx,t):=\Psi_m(\bB_m)(\bx,t)$.
These are $C^\infty$ in $\R^3\times I_k$, divergence-free, and depend only on the values of $\bw_m,\bB_m$ on $((k-2)\delta,k\delta)$. With these known coefficients, on $I_k$ we solve the linearized system
\[
\begin{cases}
\partial_t \bw_m
+(\bA_m\cdot\nabla)\bw_m
+(\bU^\bb\cdot\nabla)\bw_m
+(\bw_m\cdot\nabla)\bU^\bb
+\nabla\pi_m
=\Delta\bw_m+(\bD_m\cdot\nabla)\bB_m,\\
\partial_t \bB_m
+(\bA_m\cdot\nabla)\bB_m
+(\bU^\bb\cdot\nabla)\bB_m+\nabla \Pi_m
=\Delta\bB_m+(\bD_m\cdot\nabla)\bw_m
+(\bB_m\cdot\nabla)\bU^\bb,\\
\nabla\cdot\bw_m=0,\qquad \nabla\cdot\bB_m=0,\\
(\bw_m,\bB_m)|_{t=k\delta}=(\bw_m,\bB_m)|_{t=k\delta-0}.
\end{cases}
\]
Assume $0<|\bb|< \min\{\ve_0, \frac{1}{4C_0}\}$. By Lemma~\ref{apdx-mhd-lm}, there exists a unique solution in $C(I_k; L^2_{\sigma}(\R^3)) \cap L^2(I_k; H^1_{\sigma}(\R^3))$; pasting these solutions yields the global approximate solution, and summing the energy identities over the intervals gives
\begin{align*}
\begin{aligned}
&\|\bw_m(t)\|_2^2+\|\bB_m(t)\|_2^2
+2\int_0^t\Bigl(\|\nabla\bw_m(\tau)\|_2^2+\|\nabla\bB_m(\tau)\|_2^2\Bigr)\,\mathrm{d}\tau\\
& =\|\bw_0\|_2^2+\|\bB_0\|_2^2
-2\int_0^t\Bigl(\langle(\bw_m(\tau)\cdot\nabla)\bU^\bb,\bw_m(\tau)\rangle
-\langle(\bB_m(\tau)\cdot\nabla)\bU^\bb,\bB_m(\tau)\rangle\Bigr)\,\mathrm{d}\tau
\end{aligned}
\end{align*}
for all $t>0$. 
Using \eqref{apdx-mhd-lm-hardy}, we obtain
\[
\|\bw_m(t)\|_2^2+\|\bB_m(t)\|_2^2
+2(1-4C_0 |\bb|)\int_0^t
\Bigl(\|\nabla\bw_m(\tau)\|_2^2+\|\nabla\bB_m(\tau)\|_2^2\Bigr)\,\mathrm{d}\tau
\le \|\bw_0\|_2^2+\|\bB_0\|_2^2 .
\]
Hence, for every finite $T'>0$, the sequence $\{(\bw_m,\bB_m)\}$ is uniformly bounded in
\[
L^\infty(0,T';L^2_\sigma(\R^3))\cap L^2(0,T';H^1_\sigma(\R^3)).
\]
Since $\|\psi_\delta\|_{L^1(\R^3\times\R)}=1$, for all $1\le r,s\le\infty$, Young's convolution inequality implies
\begin{equation}\label{apdx-thm-YoungConvIneq}
\begin{aligned}
\|\bA_m\|_{L^r(0,T';L^s(\R^3))}\le\|\bw_m\|_{L^r(0,T';L^s(\R^3))}, \quad
\|\bD_m\|_{L^r(0,T';L^s(\R^3))}\le\|\bB_m\|_{L^r(0,T';L^s(\R^3))}.
\end{aligned}
\end{equation}

\underline{Step 2.}
\emph{Time derivative estimates.}
Fix an arbitrary finite $T'>0$. For any
$\boldsymbol{\zeta}\in C^\infty_c(\R^3\times(0,T'))$ with
$\nabla_{\bx}\cdot\boldsymbol{\zeta}=0$, testing the approximate
equations with $\boldsymbol{\zeta}$ yields
\begin{align}\label{apdx-thm-ApproWeakForm}
\begin{aligned}
&\int_0^{T'} \langle \bw_m, \partial_t\boldsymbol{\zeta}\rangle\,\mathrm{d}t=
\int_0^{T'}
\Bigl[
a_1(\bw_m,\boldsymbol{\zeta})
-\langle \bA_m\cdot\nabla\boldsymbol{\zeta},\bw_m\rangle
+\langle \bD_m\cdot\nabla\boldsymbol{\zeta},\bB_m\rangle
\Bigr]\mathrm{d}t,
\end{aligned}
\end{align}
with the analogous identity for $\bB_m$. Here $a_1(\bw, \bv)$ is the bilinear form defined in \eqref{eqn:bilinearform}.
We claim that
\begin{align}\label{apdx-thm-ptEstimate}
\|\partial_t\bw_m\|_{L^{4/3}(0,T';(H^1_\sigma(\R^3))^*)}
+
\|\partial_t\bB_m\|_{L^{4/3}(0,T';(H^1_\sigma(\R^3))^*)}
\le C(T').
\end{align}
Indeed, by the H\"older inequality, \eqref{eqn:bddcoercive}, and the uniform bounds from Step~1, every term
on the right-hand side of \eqref{apdx-thm-ApproWeakForm} is bounded by
$C(T')\|\boldsymbol{\zeta}\|_{L^4_t H^1_x}$.
For instance, the term involving $\bA_m$ is estimated as follows:
\[
\begin{aligned}
&\left|\int_0^{T'}\langle \bA_m\cdot\nabla\boldsymbol{\zeta},\bw_m\rangle\,\mathrm{d}t\right|
\le
\|\bA_m\|_{L^2_t L^6_x}
\|\nabla\boldsymbol{\zeta}\|_{L^4_t L^2_x}
\|\bw_m\|_{L^4_t L^3_x}\\
&
\le
\|\nabla\bw_m\|_{L^2_t L^2_x}
\|\nabla\boldsymbol{\zeta}\|_{L^4_t L^2_x}
\|\bw_m\|_{L^\infty_t L^2_x}^{1/2}
\|\bw_m\|_{L^2_t L^6_x}^{1/2}
\le
C\|\nabla\boldsymbol{\zeta}\|_{L^4_t L^2_x},
\end{aligned}
\]
where we used the interpolation inequality $\|\bw_m\|_{L^4_t L^3_x} \le \|\bw_m\|_{L^\infty_t L^2_x}^{1/2} \|\bw_m\|_{L^2_t L^6_x}^{1/2}$,
together with
\[
\|\bA_m\|_{L^2_t L^6_x}
\le
\|\bw_m\|_{L^2_t L^6_x}
\le
C\|\nabla\bw_m\|_{L^2_t L^2_x}
\le C.
\]
The $\bD_m$-terms are estimated analogously, and $a_1(\bw_m, \boldsymbol{\zeta})$ is controlled via H\"older inequality and \eqref{eqn:bddcoercive}. 
The same argument applies to the $\bB_m$-equation, so the claim follows.

\underline{Step 3.} 
\emph{Local strong convergence and weak solution.}
From Step 1, the sequence $\{(\bw_m,\bB_m)\}$ is uniformly bounded in
$L^\infty(0,T';L^2_\sigma(\R^3))\cap L^2(0,T';H^1_\sigma(\R^3))$.
Hence, up to a subsequence, we have
\[
\begin{aligned}
\bw_m \overset{*}{\rightharpoonup} \bw,
\quad
\bB_m \overset{*}{\rightharpoonup} \bB
&\quad\text{weakly-$*$ in } L^\infty(0,T';L^2_\sigma(\R^3)),\\
\bw_m \rightharpoonup \bw,
\quad
\bB_m \rightharpoonup \bB
&\quad\text{weakly in } L^2(0,T';H^1_\sigma(\R^3)).
\end{aligned}
\]
By the Aubin--Lions lemma combined with \eqref{apdx-thm-ptEstimate},
there is a further subsequence such that
\[
\bw_m \to \bw,\quad \bB_m \to \bB
\quad\text{strongly in } L^2(0,T';L^2_{\mathrm{loc}}(\R^3)).
\]
Similarly, since $\bA_m=\Psi_m(\bw_m)$ and $\bD_m=\Psi_m(\bB_m)$, the
same argument gives
\[
\bA_m \to \bw,\qquad \bD_m \to \bB
\quad\text{strongly in } L^2(0,T';L^2_{\mathrm{loc}}(\R^3)).
\]

Passing to the limit in \eqref{apdx-thm-ApproWeakForm} (and in the analogous weak form for $\bB_m$) is justified as follows. The time derivative term is handled by the weak-$*$ convergence in $L^{\infty}(0,T'; L^2_{\sigma}(\R^3))$. The terms involving $\bA_m$ and $\bD_m$ converge by the strong convergence of $\bA_m,\bD_m,\bw_m,\bB_m$ in $L^2(0,T';L^2_{\mathrm{loc}}(\R^3))$, since $\boldsymbol{\zeta}$ has compact spatial support. The convergence of $a_1(\bw_m,\boldsymbol{\zeta})$ follows from the weak convergence $\bw_m\rightharpoonup\bw$ in $L^2(0,T';H^1_\sigma(\R^3))$ and \eqref{eqn:bddcoercive}. Therefore $(\bw,\bB)$ satisfies the weak formulation of the original system.
Applying the same argument as in Step~2 to the limiting weak
formulation, we obtain
\[
\partial_t\bw,\ \partial_t\bB
\in
L^{4/3}(0,T';(H^1_\sigma(\R^3))^*).
\]
Combining this with the uniform bound in Step~1, it follows from \cite[Lemma 3.3(ii)]{Tsai2018} that, after a possible redefinition on a null set, $\bw,\bB\in C_w([0,T'];L^2_\sigma(\R^3))$.
Following a standard argument as in \cite[Lemma 3.4]{Tsai2018}, for every $\boldsymbol{\chi}\in C^\infty_{c,\sigma}(\R^3)$, we obtain
\[
\langle \bw_m(0), \boldsymbol{\chi} \rangle
\to
\langle \bw(0), \boldsymbol{\chi} \rangle
\qquad
\text{as } m\to\infty,
\]
Since
$\bw(0) \in L^2_{\sigma}(\R^3)$, we get $\bw(0)=\bw_0$. The same argument applies to
$\bB$.
Following the standard pressure recovery argument in \cite[pp.~55, (3.13)--(3.15)]{Tsai2018} for both the velocity and magnetic equations, the pressures $\pi_m$ and $\Pi_m$ are well-defined as distributions.

\underline{Step 4.} \emph{Global strong convergence.}
Fix an arbitrary finite $T'>0$.
We first derive the equations for the pressure terms. Taking the divergence of the approximate velocity and magnetic equations and using the divergence-free conditions, we obtain
\[
\begin{aligned}
-\Delta \pi_m
&=
\partial_{ij}\Bigl[
(\bA_m)_i(\bw_m)_j
+
(\bU^{\bb})_i(\bw_m)_j
+
(\bw_m)_i(\bU^{\bb})_j
-
(\bD_m)_i(\bB_m)_j
\Bigr],\\
-\Delta \Pi_m
&=
\partial_{ij}\Bigl[
(\bA_m)_i(\bB_m)_j
-
(\bD_m)_i(\bw_m)_j
+
(\bU^{\bb})_i(\bB_m)_j
-
(\bB_m)_i(\bU^{\bb})_j
\Bigr].
\end{aligned}
\]
We define the candidate pressures by
\[
\begin{aligned}
\pi_m^{\mathrm{nl}}
&= R_iR_j\bigl[
(\bA_m)_i(\bw_m)_j
-(\bD_m)_i(\bB_m)_j
\bigr],
\quad
\pi_m^{\mathrm{lin}}
= R_iR_j\bigl[
(\bU^{\bb})_i(\bw_m)_j
+ (\bw_m)_i(\bU^{\bb})_j
\bigr],\\
\Pi_m^{\mathrm{nl}}
&= R_iR_j\bigl[
(\bA_m)_i(\bB_m)_j
- (\bD_m)_i(\bw_m)_j
\bigr],
\quad
\Pi_m^{\mathrm{lin}}
= R_iR_j\bigl[
(\bU^{\bb})_i(\bB_m)_j
- (\bB_m)_i(\bU^{\bb})_j
\bigr],
\end{aligned}
\]
where $R_i$ denotes the Riesz transform. Since the tensor $(\bU^{\bb})_i(\bB_m)_j-(\bB_m)_i(\bU^{\bb})_j$ is antisymmetric in $i,j$ and $\partial_{ij}$ is symmetric, we have $\Pi_m^{\mathrm{lin}}=0$.

We now estimate these candidate pressures. By the boundedness of Riesz transforms on $L^q(\mathbb R^3)$ for $1<q<\infty$ (see \cite[Theorem 1.3]{Tsai2018}), it suffices to estimate the corresponding products. For the nonlinear parts, the uniform bound from Step 1 and \eqref{apdx-thm-YoungConvIneq} yield
\[
\|(\bA_m)_i (\bw_m)_j\|_{L^{5/3}_{t,x}} 
\le \|(\bA_m)_i \|_{L^{10/3}_{t,x}}
\|(\bw_m)_j\|_{L^{10/3}_{t,x}},
\]
and similarly for the other nonlinear terms. For the linear terms, using \eqref{eqn:LandauEst} and Hardy's inequality \eqref{eqn:Hardy}, we get
\[
\|(\bU^\bb)_i(\bw_m)_j\|_{2}
\le C_0|\bb|\,\Bigl\|\frac{\bw_m}{|\cdot|}\Bigr\|_2
\le 2C_0|\bb|\,\|\nabla\bw_m\|_2,
\]
and the same bound holds for $(\bw_m)_i(\bU^\bb)_j$. Therefore,
\begin{align}\label{apdx-thm-pressureEst}
\|\pi_m^{\mathrm{nl}}\|_{L^{5/3}(0,T';L^{5/3}(\R^3))} \le C,\quad
\|\pi_m^{\mathrm{lin}}\|_{L^2(0,T';L^2(\R^3))} \le C,\quad
\|\Pi_m^{\mathrm{nl}}\|_{L^{5/3}(0,T';L^{5/3}(\R^3))} \le C.
\end{align}
These integrability estimates confirm that the candidate pressures defined above are indeed the pressure terms of the approximate system; see \cite[p.~65, Step~4]{Tsai2018}.

For the tail estimate, we choose $\phi_R(\bx)=1-\phi(\bx/R)$, where $\phi\in C^\infty_c(\R^3)$,
$0\le\phi\le1$, $\phi=1$ on $B_1$, and
$\operatorname{supp}\phi\subset B_2$. Then
$|\nabla\phi_R|\le C/R$, and
$\operatorname{supp}\nabla\phi_R\subset\{R\le|\bx|\le2R\}$.
Testing the approximate velocity and magnetic equations with $\phi_R^2\bw_m(t)$ and $\phi_R^2\bB_m(t)$, respectively, we obtain
\begin{align*}
&\frac12 \frac{\mathrm{d}}{\mathrm{d}t} \|\bw_m \phi_R\|_2^2 
+ a_1(\bw_m, \bw_m \phi_R^2) - \langle \pi_m, \bw_m \cdot \nabla (\phi_R^2)  \rangle \\
&\qquad + \langle (\bA_m \cdot \nabla) \bw_m, \bw_m \phi_R^2 \rangle 
- \langle (\bD_m \cdot \nabla) \bB_m, \bw_m \phi_R^2 \rangle = 0, \\
&\frac12 \frac{\mathrm{d}}{\mathrm{d}t} \|\bB_m \phi_R\|_2^2 
+ a_2(\bB_m, \bB_m \phi_R^2) 
- \langle \Pi_m, \bB_m \cdot \nabla (\phi_R^2)  \rangle\\
&\qquad + \langle (\bA_m \cdot \nabla) \bB_m, \bB_m \phi_R^2 \rangle 
- \langle (\bD_m \cdot \nabla) \bw_m, \bB_m \phi_R^2 \rangle = 0 .
\end{align*}
Here we used $\langle \partial_t\bw_m,\phi_R^2\bw_m\rangle = \frac12\frac{\mathrm{d}}{\mathrm{d}t}\|\phi_R\bw_m\|_2^2$, and the same holds for $\bB_m$.
Summing together yields
\begin{align}\label{apdx-thm-globalIdentity}
\begin{aligned}
&\frac12 \frac{\mathrm{d}}{\mathrm{d}t} \bigl( \|\bw_m \phi_R\|_2^2 + \|\bB_m \phi_R\|_2^2 \bigr)
- \langle \pi_m, \bw_m \cdot \nabla \phi_R^2  \rangle
- \langle \Pi_m, \bB_m \cdot \nabla \phi_R^2  \rangle\\ 
&\quad + a_1(\bw_m, \bw_m \phi_R^2) + a_2(\bB_m, \bB_m \phi_R^2) 
+ \mathcal{N} = 0,
\end{aligned}
\end{align}
where 
\begin{align*}
\mathcal{N} =& \langle (\bA_m \cdot \nabla) \bw_m, \bw_m \phi_R^2 \rangle 
- \langle (\bD_m \cdot \nabla) \bB_m, \bw_m \phi_R^2 \rangle \\
&+ \langle (\bA_m \cdot \nabla) \bB_m, \bB_m \phi_R^2 \rangle 
- \langle (\bD_m \cdot \nabla) \bw_m, \bB_m \phi_R^2 \rangle .  
\end{align*}
A direct computation shows
\begin{align*}
\langle (\bD_m \cdot \nabla) \bB_m, \bw_m \phi_R^2 \rangle +  \langle (\bD_m \cdot \nabla) \bw_m, \bB_m \phi_R^2 \rangle
=& \int_{\R^3} (\phi_R^2 \, \bD_m) \cdot \nabla (\bB_m \cdot \bw_m) \\
=& -\int_{\R^3}  (\bD_m\cdot \nabla \phi_R^2) (\bB_m \cdot \bw_m).
\end{align*}
Similarly, we have
\begin{align*}
\langle \bA_m \cdot \nabla \bw_m, \bw_m \phi_R^2 \rangle
+\langle \bA_m \cdot \nabla \bB_m, \bB_m \phi_R^2 \rangle
= -\frac12 \int_{\R^3} (|\bw_m|^2+|\bB_m|^2) (\bA_m \cdot \nabla \phi_R^2).
\end{align*}
Hence
\begin{align}\label{apdx-thm-mathcalNId}
\mathcal{N} 
= -\frac12 \int_{\R^3} (|\bw_m|^2 + |\bB_m|^2) \, \bA_m \cdot \nabla \phi_R^2 
+ \int_{\R^3} (\bB_m \cdot \bw_m) \, \bD_m \cdot \nabla \phi_R^2 .
\end{align}
For $a_1(\bw_m, \bw_m \phi_R^2)$, one can verify that
\begin{align*}
&a_1(\bw_m,\bw_m \phi_R^2)\\
&=\|\nabla(\phi_R \bw_m)\|_2^2 
    - \int_{\mathbb R^3} |\bw_m|^2 |\nabla \phi_R|^2
    - \frac12 \int_{\R^3} |\bw_m|^2 \bU^{\bb} \cdot \nabla \phi_R^2 
    +\langle \phi_R\bw_m \cdot \nabla \bU^\bb, \phi_R\bw_m \rangle.
\end{align*}
Here  we used  
\begin{align*}
\langle \nabla \bw_{m},\nabla (\bw_{m}\phi_{R}^{2})\rangle 
&=\|\nabla(\phi_R \bw_m)\|_2^2 
- \int_{\mathbb R^3} |\bw_m|^2 |\nabla \phi_R|^2,
\end{align*}
and 
\[\langle \bU^{\bb} \cdot \nabla \bw_m, \bw_m \phi_R^2 \rangle=- \frac12 \int_{\R^3} |\bw_m|^2 (\bU^{\bb} \cdot \nabla \phi_R^2).\]
On $\{\bx: R \le |\bx| \le 2R \}$, we have $|\bU^{\bb}| \le \frac{C_0|\bb|}{|\bx|} \le \frac{C_0|\bb|}{R}$, which yields
\[
\left|- \frac12 \int_{\R^3} |\bw_m|^2 \bU^{\bb} \cdot \nabla \phi_R^2 \right|
\le \frac{C}{R} 
\int_{\{ R \le |\bx| \le 2R \}} \frac{|\bw_m|^2}{|\bx|}
\le \frac{C}{R^2} \|\bw_m\|_2^2.
\]
Combining this with \eqref{apdx-mhd-lm-hardy}, we obtain 
\[
a_1(\bw_m,\bw_m \phi_R^2) \ge (1-4C_0|\bb|) \|\nabla(\phi_R \bw_m)\|_2^2 - \frac{C}{R^2},
\]
and similarly,
\[
a_2(\bB_m,\bB_m \phi_R^2) \ge (1-4C_0|\bb|) \|\nabla(\phi_R \bB_m)\|_2^2 - \frac{C}{R^2}.
\]
Next we estimate $\mathcal{N}$.
Interpolation between $L^\infty(0,T'; L^2(\R^3))$ and $L^2(0,T';H^1_{\sigma}(\R^3))$, together with the finiteness of $T'$, implies $\bw_m,\bB_m\in L^3(0,T'; L^3(\R^3))$ uniformly in $m$. 
Thus, \eqref{apdx-thm-mathcalNId} becomes
\[
\begin{aligned}
\int_0^{T'}\left|\mathcal{N}\right|\,\mathrm{d}t 
&\le
\frac{C}{R}
\Bigl(
(\|\bw_m\|_{L^3(0,T'; L^3(\R^3))}^2+
\|\bB_m\|_{L^3(0,T'; L^3(\R^3))}^2)\|\bA_m\|_{L^3(0,T'; L^3(\R^3))}\\
&\qquad + \|\bB_m\|_{L^3(0,T'; L^3(\R^3))}
\|\bw_m\|_{L^3(0,T'; L^3(\R^3))}
\|\bD_m\|_{L^3(0,T'; L^3(\R^3))}
\Bigr)
\le
\frac{C_{T'}}{R},
\end{aligned}
\]
where the uniform $L^3(0,T'; L^3(\R^3))$-bounds for $\bA_m$ and $\bD_m$ follow from \eqref{apdx-thm-YoungConvIneq}, with norms controlled by those of $\bw_m$ and $\bB_m$, respectively.
The pressure terms are estimated as follows.
With the decomposition
$\pi_m=\pi_m^{\mathrm{nl}}+\pi_m^{\mathrm{lin}}$ and  \eqref{apdx-thm-pressureEst}, 
H\"older inequality gives
\[
\begin{aligned}
\int_0^{T'}\left|\int_{\R^3}
\pi_m\,\bw_m\cdot\nabla\phi_R^2\right|\, \mathrm{d}t 
&\le
\frac{C}{R}
\bigl(
\|\pi_m^{\mathrm{nl}}\|_{L^{5/3}(0,T'; L^{5/3}(\R^3))}
\|\bw_m\|_{L^{5/2}(0,T'; L^{5/2}(\R^3))}\\
&\qquad +
\|\pi_m^{\mathrm{lin}}\|_{L^{2}(0,T'; L^{2}(\R^3))}
\|\bw_m\|_{L^{2}(0,T'; L^{2}(\R^3))}
\bigr)
\le
\frac{C_{T'}}{R},
\end{aligned}
\]
where we used $\bw_m\in L^{5/2}(0,T'; L^{5/2}(\R^3))$ uniformly in $m$, which follows from $\bw_m\in L^{20/3}(0,T'; L^{5/2}(\R^3))$ and the finiteness of $T'$. The estimate for $\langle \Pi_m, \bB_m \cdot \nabla \phi_R^2 \rangle$ is identical.

Integrating \eqref{apdx-thm-globalIdentity} over $[0,t]$ and using
$\bw_m(0)=\bw_0$, $\bB_m(0)=\bB_0$, and the estimates obtained above, we obtain
\[
\begin{aligned}
&\|\bw_m(t)\phi_R\|_2^2
+\|\bB_m(t)\phi_R\|_2^2
+2(1-4C_0|\bb|)\int_0^{t}
\Bigl(\|\nabla(\phi_R\bw_m)\|_2^2
+\|\nabla(\phi_R\bB_m)\|_2^2\Bigr)\,\mathrm{d}\tau\\
&\quad
\le
\int_{|\bx|\ge R}\bigl(|\bw_0|^2+|\bB_0|^2\bigr)\, \mathrm{d}\bx
+\frac{C_{T'}}{R}
+\frac{C_{T'}}{R^2}.
\end{aligned}
\]
Since $4C_0|\bb|<1$, the gradient terms are nonnegative. Hence, for
every $\varepsilon>0$, there exists $R_1>0$ such that for all $R\ge R_1$,
\begin{equation}\label{apdx-thm-TailEst}
\sup_m\sup_{0\le t\le T'}
\bigl(\|\phi_R\bw_m(t)\|_2^2+\|\phi_R\bB_m(t)\|_2^2\bigr)
\le \frac{\varepsilon}{4}.
\end{equation}
Now fix $R\ge R_1$. For the velocity field,
\[
\begin{aligned}
\|\bw_m-\bw\|_2^2
&\le
2\|\phi_R(\bw_m-\bw)\|_2^2
+
2\|(1-\phi_R)(\bw_m-\bw)\|_2^2\\
&\le
4\|\phi_R\bw_m\|_2^2
+
4\|\phi_R\bw\|_2^2
+
2\|(1-\phi_R)(\bw_m-\bw)\|_2^2.
\end{aligned}
\]
Integrating in time gives
\[
\int_0^{T'}\|\bw_m-\bw\|_2^2\,\mathrm{d}t
\le
4\int_0^{T'}\|\phi_R\bw_m\|_2^2 \,\mathrm{d}t
+4\int_0^{T'} \|\phi_R\bw\|_2^2\,\mathrm{d}t
+
2\int_0^{T'}\|(1-\phi_R)(\bw_m-\bw)\|_2^2\,\mathrm{d}t.
\]
For the first term, by \eqref{apdx-thm-TailEst},
\[
\int_0^{T'} 4\|\phi_R\bw_m\|_2^2\,\mathrm{d}t
\le
4\int_0^{T'}\frac{\varepsilon}{4}\,\mathrm{d}t
=\varepsilon T'.
\]
The finiteness of $T'$ implies $\bw\in L^2(0,T';L^2_\sigma(\R^3))$, hence there is $R_2>0$ such that for all $R\ge R_2$,
\[
\int_0^{T'}\|\phi_R\bw\|_2^2\,\mathrm{d}t
\le
\int_0^{T'}\int_{|x|\ge R}|\bw|^2\,\mathrm{d}x\,\mathrm{d}t
\le \frac{\ve}{2} .
\]
Finally, since $1-\phi_R$ is supported in $\{|\bx|\le 2R\}$, the local strong convergence from Step~3 gives
\[
2\int_0^{T'}\|(1-\phi_R)(\bw_m-\bw)\|_2^2\,\mathrm{d}t
\le
2\int_0^{T'}\int_{|x|\le 2R}|\bw_m-\bw|^2\,\mathrm{d}\bx\,\mathrm{d}t
\le \frac{\ve}{2},
\]
for all $m$ sufficiently large.
Combining the estimates,
\[
\int_0^{T'}\|\bw_m-\bw\|_2^2\, \mathrm{d}t
\le
\varepsilon(1+T').
\]
The same argument applies equally to $\bB_m-\bB$. Since
$\varepsilon>0$ is arbitrary and $T'$ is finite, we conclude
\[
\bw_m\to\bw,\quad \bB_m\to\bB
\quad\text{strongly in }L^2(0,T';L^2(\R^3)).
\]
As $T'>0$ is arbitrary finite, a diagonal argument gives the desired
convergence on every finite subinterval of $(0,\infty)$.

We next verify that no pressure gradient appears in the limiting magnetic equation. For the approximate magnetic pressure, the Poisson equation is
$$-\Delta\Pi_m
=
\partial_{ij}\bigl(
(A_m)_i(B_m)_j-(D_m)_i(w_m)_j
\bigr)
=:
\partial_{ij}F_{ij,m}.$$
Since $A_m,D_m,w_m,B_m$ converge strongly in $L^2_{t,x}$, we have
$F_{ij,m}\to w_iB_j-B_iw_j$ distributionally; this limiting tensor is antisymmetric in $i,j$, while $\partial_{ij}$ is symmetric, hence $\partial_{ij}F_{ij,m}\to0$ distributionally. 
On the other hand, the uniform $L^{5/3}_{t,x}$-bound yields, up to a subsequence, a weak limit $\Pi_m\rightharpoonup\Pi$ in $L^{5/3}_{t,x}$, so $-\Delta\Pi=0$. 
Since for almost every $t$ the function $\Pi(\cdot,t)\in L^{5/3}(\mathbb R^3)$ is harmonic, and the only $L^{5/3}(\mathbb R^3)$-integrable harmonic function is zero, we get $\Pi=0$. Therefore $\Pi_m\rightharpoonup0$ weakly in $L^{5/3}_{t,x}$, and the magnetic pressure disappears in the limit.

\underline{Step 5.} \emph{Strong energy inequality.}
The approximate solutions satisfy the energy identity
\[
\begin{aligned}
\|\bw_m(t)\|_2^2+\|\bB_m(t)\|_2^2
+2\int_s^t\bigl(a_1(\bw_m,\bw_m)+a_2(\bB_m,\bB_m)\bigr)\,\mathrm{d}\tau
=\|\bw_m(s)\|_2^2+\|\bB_m(s)\|_2^2,
\end{aligned}
\]
where $a_1,a_2$ are defined in \eqref{eqn:bilinearform}.
By the Hardy estimates \eqref{eqn:bddcoercive} and $4C_0|\bb|<1$,
\[
(1-4C_0|\bb|)\|\nabla \bu\|_2^2
\le a_i(\bu,\bu)
\le (1+4C_0|\bb|)\|\nabla \bu\|_2^2,
\qquad i=1,2.
\]
Since the quadratic form $a_i(\bu,\bu)$ depends only on the symmetric part
of $a_i$ and $4C_0|\bb|<1$,  it follows from the above inequality that $a_i(\cdot,\cdot)$ defines an equivalent Hilbertian norm on $\dot H^1_\sigma(\R^3)$.
Let $Y_{a_i}$ denote the space
$\dot{H}^1_\sigma(\R^3)$ endowed with this norm.
Then the norm
\[
\|\bu\|_{L^2(s,t;Y_{a_i})}^2
:=
\int_s^t a_i(\bu(\tau), \bu(\tau))\,\mathrm{d}\tau
\]
is equivalent to the standard norm on $L^2(s,t;\dot H^1_\sigma(\R^3))$. 
Therefore weak convergence in $L^2(s,t;\dot H^1_\sigma(\R^3))$ is unchanged under this renorming.
By the weak lower semicontinuity of the norm, we get
\begin{align}\label{apdx-thm-lwsemicontinuity}
\int_s^t a_i(\bu,\bu)\,\mathrm{d}\tau
\le
\liminf_{m\to\infty}
\int_s^t a_i(\bu_m,\bu_m)\,\mathrm{d}\tau.
\end{align}

From the global strong convergence in Step 4, along a subsequence we have
\[
\bw_m(t)\to\bw(t),\qquad \bB_m(t)\to\bB(t)
\quad\text{strongly in }L^2(\R^3)
\]
for every $t\notin\Sigma$, where $\Sigma\subset[0,T']$ has measure zero.
For fixed $s,t\notin\Sigma$, $s<t$, passing to the limit in the energy identity and using \eqref{apdx-thm-lwsemicontinuity}, we get
\[
\begin{aligned}
\|\bw(t)\|_2^2+\|\bB(t)\|_2^2
+2\int_s^t\bigl(a_1(\bw,\bw)+a_2(\bB,\bB)\bigr)\,\mathrm{d}\tau
\le
\|\bw(s)\|_2^2+\|\bB(s)\|_2^2.
\end{aligned}
\]
For fixed $s\notin\Sigma$, let $t>s$ be arbitrary and choose $t_k\notin\Sigma$ with $t_k\to t$. By weak continuity of $\bw,\bB$, we get
\begin{align*}
\begin{aligned}
&\|\bw(t)\|_2^2+\|\bB(t)\|_2^2
+2\int_s^t(a_1(\bw,\bw)+a_2(\bB,\bB))\,\mathrm{d}\tau\\
&\le
\liminf_{k\to\infty}
\Big[
\|\bw(t_k)\|_2^2+\|\bB(t_k)\|_2^2
+2\int_s^{t_k}(a_1(\bw,\bw)+a_2(\bB,\bB))\,\mathrm{d}\tau
\Big]
\le
\|\bw(s)\|_2^2+\|\bB(s)\|_2^2.
\end{aligned}
\end{align*}
Hence the strong energy inequality holds for every $s\notin\Sigma$ and all $t>s$.
For $s=0$, applying the previous limiting procedure to the energy identity on $[0,t]$ for $t\notin\Sigma$ yields the strong energy inequality with $s=0$, which extends to all $t>0$ by the same weak-continuity argument as above.
Since $T'>0$ is arbitrary, the proof is complete.
\end{proof}

\vspace{0.5cm}

\printbibliography

@misc{Zhang2025MHDwithoutMD,
      title={Axisymmetric self-similar solutions to the MHD equations without magnetic diffusion}, 
      author={Shaoheng Zhang},
      year={2025},
      eprint={2506.20131},
      archivePrefix={arXiv},
      primaryClass={math.AP},
      url={https://arxiv.org/abs/2506.20131}, 
}

@article{Zhang2025MHDwithMD,
 author = {Zhang, Shaoheng},
 title = {On axisymmetric self-similar solutions to the {MHD} system},
 fjournal = {Bulletin of the Malaysian Mathematical Sciences Society. Second Series},
 journal = {Bull. Malays. Math. Sci. Soc. (2)},
 issn = {0126-6705},
 volume = {49},
 number = {2},
 pages = {8},
 note = {Id/No 75},
 year = {2026},
 language = {English},
 doi = {10.1007/s40840-026-02076-8},
 zbMATH = {8181229}
}

@book {Davidson_Book_2nd_2017,
    AUTHOR = {Davidson, P. A.},
     TITLE = {Introduction to magnetohydrodynamics},
    SERIES = {Cambridge Texts in Applied Mathematics},
   EDITION = {Second},
 PUBLISHER = {Cambridge University Press, Cambridge},
      YEAR = {2017},
     PAGES = {xviii+555},
      ISBN = {978-1-316-61302-3},
   MRCLASS = {76W05 (76-01)},
  MRNUMBER = {3699477},
MRREVIEWER = {J. W. Jerome},
       DOI = {10.1017/9781316672853},
       URL = {https://doi.org/10.1017/9781316672853},
}

@article {KarchPilarczyk2011,
    AUTHOR = {Karch, Grzegorz and Pilarczyk, Dominika},
     TITLE = {Asymptotic stability of {L}andau solutions to
              {N}avier-{S}tokes system},
   JOURNAL = {Arch. Ration. Mech. Anal.},
  FJOURNAL = {Archive for Rational Mechanics and Analysis},
    VOLUME = {202},
      YEAR = {2011},
    NUMBER = {1},
     PAGES = {115--131},
      ISSN = {0003-9527},
   MRCLASS = {76D05 (35B35 35Q30)},
  MRNUMBER = {2835864},
MRREVIEWER = {Jos\'{e} Luiz Boldrini},
       DOI = {10.1007/s00205-011-0409-z},
       URL = {https://doi.org/10.1007/s00205-011-0409-z},
}

@article {Landau1944,
    AUTHOR = {Landau, L.},
     TITLE = {A new exact solution of {N}avier-{S}tokes equations},
   JOURNAL = {C. R. (Doklady) Acad. Sci. URSS (N.S.)},
  FJOURNAL = {C. R. (Doklady) Acad. Sci. URSS (N.S.)},
    VOLUME = {43},
      YEAR = {1944},
     PAGES = {286--288},
   MRCLASS = {76.1X},
  MRNUMBER = {11205},
MRREVIEWER = {C. C. Torrance},
}

@article {CannoneKarch2004,
    AUTHOR = {Cannone, Marco and Karch, Grzegorz},
     TITLE = {Smooth or singular solutions to the {N}avier-{S}tokes system?},
   JOURNAL = {J. Differential Equations},
  FJOURNAL = {Journal of Differential Equations},
    VOLUME = {197},
      YEAR = {2004},
    NUMBER = {2},
     PAGES = {247--274},
      ISSN = {0022-0396},
   MRCLASS = {35Q30 (76D03 76D05)},
  MRNUMBER = {2034160},
MRREVIEWER = {Bruno J. Scarpellini},
       DOI = {10.1016/j.jde.2003.10.003},
       URL = {https://doi.org/10.1016/j.jde.2003.10.003},
}

@article {TianXin1998,
    AUTHOR = {Tian, Gang and Xin, Zhouping},
     TITLE = {One-point singular solutions to the {N}avier-{S}tokes
              equations},
   JOURNAL = {Topol. Methods Nonlinear Anal.},
  FJOURNAL = {Topological Methods in Nonlinear Analysis},
    VOLUME = {11},
      YEAR = {1998},
    NUMBER = {1},
     PAGES = {135--145},
      ISSN = {1230-3429},
   MRCLASS = {35Q30 (76D05)},
  MRNUMBER = {1642049},
MRREVIEWER = {Zhi Min Chen},
       DOI = {10.12775/TMNA.1998.008},
       URL = {https://doi.org/10.12775/TMNA.1998.008},
}

@book {Tsai2018,
    AUTHOR = {Tsai, Tai-Peng},
     TITLE = {Lectures on {N}avier-{S}tokes equations},
    SERIES = {Graduate Studies in Mathematics},
    VOLUME = {192},
 PUBLISHER = {American Mathematical Society, Providence, RI},
      YEAR = {2018},
     PAGES = {xii+224},
      ISBN = {978-1-4704-3096-2},
   MRCLASS = {35Q30 (35Q35 76D05)},
  MRNUMBER = {3822765},
MRREVIEWER = {Nader Masmoudi},
       DOI = {10.1090/gsm/192},
       URL = {https://doi.org/10.1090/gsm/192},
}

@article{Sverak2011,
 author = {Sver{\'a}k, V.},
 title = {On {Landau}'s solutions of the {Navier}-{Stokes} equations},
 fjournal = {Journal of Mathematical Sciences (New York)},
 journal = {J. Math. Sci., New York},
 issn = {1072-3374},
 volume = {179},
 number = {1},
 pages = {208--228},
 year = {2011},
 language = {English},
 doi = {10.1007/s10958-011-0590-5},
 zbMATH = {6309074},
 Zbl = {1291.35193}
}

@book {Lions61,
    AUTHOR = {Lions, J.-L.},
     TITLE = {\'Equations diff\'erentielles op\'erationnelles et probl\`emes
              aux limites},
    SERIES = {Die Grundlehren der mathematischen Wissenschaften},
    VOLUME = {Band 111},
 PUBLISHER = {Springer-Verlag, Berlin-G\"ottingen-Heidelberg},
      YEAR = {1961},
     PAGES = {ix+292},
   MRCLASS = {35.00 (34.95)},
  MRNUMBER = {153974},
MRREVIEWER = {S.\ Zaidman},
}

@article{ZhangWangWang2026,
 author = {Zhang, Shaoheng and Wang, Kui and Wang, Yun},
 title = {Point singularities of solutions to the stationary incompressible {MHD} equations},
 fjournal = {Journal of Differential Equations},
 journal = {J. Differ. Equations},
 issn = {0022-0396},
 volume = {454},
 pages = {18},
 note = {Id/No 113935},
 year = {2026},
 language = {English},
 doi = {10.1016/j.jde.2025.113935},
 zbMATH = {8143097},
 Zbl = {1580.35188}
}

@article {HishidaSchonbek16,
    AUTHOR = {Hishida, Toshiaki and Schonbek, Maria E.},
     TITLE = {Stability of time-dependent {N}avier-{S}tokes flow and
              algebraic energy decay},
   JOURNAL = {Indiana Univ. Math. J.},
  FJOURNAL = {Indiana University Mathematics Journal},
    VOLUME = {65},
      YEAR = {2016},
    NUMBER = {4},
     PAGES = {1307--1346},
      ISSN = {0022-2518,1943-5258},
   MRCLASS = {35Q30 (35B35 35B40 76D05)},
  MRNUMBER = {3549203},
MRREVIEWER = {Pavel\ I.\ Naumkin},
       DOI = {10.1512/iumj.2016.65.5843},
       URL = {https://doi.org/10.1512/iumj.2016.65.5843},
}

@book {BerghLofstrom76,
    AUTHOR = {Bergh, J\"oran and L\"ofstr\"om, J\"orgen},
     TITLE = {Interpolation spaces. {A}n introduction},
    SERIES = {Grundlehren der Mathematischen Wissenschaften},
    VOLUME = {No. 223},
 PUBLISHER = {Springer-Verlag, Berlin-New York},
      YEAR = {1976},
     PAGES = {x+207},
   MRCLASS = {46M35},
  MRNUMBER = {482275},
}

@article {Leray1934,
    AUTHOR = {Leray, Jean},
     TITLE = {Sur le mouvement d'un liquide visqueux emplissant l'espace},
   JOURNAL = {Acta Math.},
  FJOURNAL = {Acta Mathematica},
    VOLUME = {63},
      YEAR = {1934},
    NUMBER = {1},
     PAGES = {193--248},
      ISSN = {0001-5962,1871-2509},
   MRCLASS = {99-04},
  MRNUMBER = {1555394},
       DOI = {10.1007/BF02547354},
       URL = {https://doi.org/10.1007/BF02547354},
}

@article {BorchersMiyakawa92,
    AUTHOR = {Borchers, Wolfgang and Miyakawa, Tetsuro},
     TITLE = {{$L^2$}-decay for {N}avier-{S}tokes flows in unbounded
              domains, with application to exterior stationary flows},
   JOURNAL = {Arch. Rational Mech. Anal.},
  FJOURNAL = {Archive for Rational Mechanics and Analysis},
    VOLUME = {118},
      YEAR = {1992},
    NUMBER = {3},
     PAGES = {273--295},
      ISSN = {0003-9527},
   MRCLASS = {35Q30 (76D05)},
  MRNUMBER = {1158939},
       DOI = {10.1007/BF00387899},
       URL = {https://doi.org/10.1007/BF00387899},
}

@article{zbMATH08005764,
 author = {Sammartino, M. and Schonbek, M. E. and Sciacca, V.},
 title = {Dissipative 2D {MHD} equations with {{\(L^1\)}} vorticity and magnetic current},
 fjournal = {Journal of Hyperbolic Differential Equations},
 journal = {J. Hyperbolic Differ. Equ.},
 issn = {0219-8916},
 volume = {21},
 number = {3},
 pages = {791--810},
 year = {2024},
 doi = {10.1142/S0219891624400083},
 url = {hdl.handle.net/10447/672323},
 zbMATH = {8005764}
}

@article {MR3780143,
    AUTHOR = {Cai, Yuan and Lei, Zhen},
     TITLE = {Global well-posedness of the incompressible
              magnetohydrodynamics},
   JOURNAL = {Arch. Ration. Mech. Anal.},
  FJOURNAL = {Archive for Rational Mechanics and Analysis},
    VOLUME = {228},
      YEAR = {2018},
    NUMBER = {3},
     PAGES = {969--993},
      ISSN = {0003-9527,1432-0673},
   MRCLASS = {76W05 (35B30 35Q35)},
  MRNUMBER = {3780143},
MRREVIEWER = {V.\ D.\ Sharma},
       DOI = {10.1007/s00205-017-1210-4},
       URL = {https://doi.org/10.1007/s00205-017-1210-4},
}

@article {MR1616418,
    AUTHOR = {Wu, Jiahong},
     TITLE = {Viscous and inviscid magnetohydrodynamics equations},
   JOURNAL = {J. Anal. Math.},
  FJOURNAL = {Journal d'Analyse Math\'{e}matique},
    VOLUME = {73},
      YEAR = {1997},
     PAGES = {251--265},
      ISSN = {0021-7670,1565-8538},
   MRCLASS = {35Q35 (76W05)},
  MRNUMBER = {1616418},
       DOI = {10.1007/BF02788146},
       URL = {https://doi.org/10.1007/BF02788146},
}

@article {MR346289,
    AUTHOR = {Duvaut, G. and Lions, J.-L.},
     TITLE = {In\'{e}quations en thermo\'{e}lasticit\'{e} et
              magn\'{e}tohydrodynamique},
   JOURNAL = {Arch. Rational Mech. Anal.},
  FJOURNAL = {Archive for Rational Mechanics and Analysis},
    VOLUME = {46},
      YEAR = {1972},
     PAGES = {241--279},
      ISSN = {0003-9527},
   MRCLASS = {35B45 (73.35 76.35)},
  MRNUMBER = {346289},
MRREVIEWER = {P.\ Germain},
       DOI = {10.1007/BF00250512},
       URL = {https://doi.org/10.1007/BF00250512},
}

@article {MR716200,
    AUTHOR = {Sermange, Michel and Temam, Roger},
     TITLE = {Some mathematical questions related to the {MHD} equations},
   JOURNAL = {Comm. Pure Appl. Math.},
  FJOURNAL = {Communications on Pure and Applied Mathematics},
    VOLUME = {36},
      YEAR = {1983},
    NUMBER = {5},
     PAGES = {635--664},
      ISSN = {0010-3640,1097-0312},
   MRCLASS = {76W05 (35Q10)},
  MRNUMBER = {716200},
MRREVIEWER = {P.\ L.\ Sulem},
       DOI = {10.1002/cpa.3160360506},
       URL = {https://doi.org/10.1002/cpa.3160360506},
}

@article {MR3766969,
    AUTHOR = {Tan, Zhong and Wang, Yanjin},
     TITLE = {Global well-posedness of an initial-boundary value problem for
              viscous non-resistive {MHD} systems},
   JOURNAL = {SIAM J. Math. Anal.},
  FJOURNAL = {SIAM Journal on Mathematical Analysis},
    VOLUME = {50},
      YEAR = {2018},
    NUMBER = {1},
     PAGES = {1432--1470},
      ISSN = {0036-1410,1095-7154},
   MRCLASS = {76W05 (35B30 35Q35 76N10)},
  MRNUMBER = {3766969},
MRREVIEWER = {Iuliana\ Oprea},
       DOI = {10.1137/16M1088156},
       URL = {https://doi.org/10.1137/16M1088156},
}

@book {Galdi11,
    AUTHOR = {Galdi, G. P.},
     TITLE = {An introduction to the mathematical theory of the
              {N}avier-{S}tokes equations},
    SERIES = {Springer Monographs in Mathematics},
   EDITION = {Second},
 PUBLISHER = {Springer, New York},
      YEAR = {2011},
     PAGES = {xiv+1018},
      ISBN = {978-0-387-09619-3},
   MRCLASS = {35Q30 (35-02 76D03 76D05 76D07)},
  MRNUMBER = {2808162},
       DOI = {10.1007/978-0-387-09620-9},
       URL = {https://doi.org/10.1007/978-0-387-09620-9},
}

@book {Temam1979,
    AUTHOR = {Temam, Roger},
     TITLE = {Navier-{S}tokes equations},
    SERIES = {Studies in Mathematics and its Applications},
    VOLUME = {2},
   EDITION = {Revised},
      NOTE = {Theory and numerical analysis,
              With an appendix by F. Thomasset},
 PUBLISHER = {North-Holland Publishing Co., Amsterdam-New York},
      YEAR = {1979},
     PAGES = {x+519},
      ISBN = {0-444-85307-3},
   MRCLASS = {35Q10 (49D99 65P05 76D05)},
  MRNUMBER = {603444},
}

@article {Kato1984,
    AUTHOR = {Kato, Tosio},
     TITLE = {Strong {$L\sp{p}$}-solutions of the {N}avier-{S}tokes equation
              in {${\bf R}\sp{m}$}, with applications to weak solutions},
   JOURNAL = {Math. Z.},
  FJOURNAL = {Mathematische Zeitschrift},
    VOLUME = {187},
      YEAR = {1984},
    NUMBER = {4},
     PAGES = {471--480},
      ISSN = {0025-5874,1432-1823},
   MRCLASS = {35Q10 (76D05)},
  MRNUMBER = {760047},
MRREVIEWER = {Yoshikazu\ Giga},
       DOI = {10.1007/BF01174182},
       URL = {https://doi.org/10.1007/BF01174182},
}

@article {Masuda1984,
    AUTHOR = {Masuda, Ky\=uya},
     TITLE = {Weak solutions of {N}avier-{S}tokes equations},
   JOURNAL = {Tohoku Math. J. (2)},
  FJOURNAL = {The Tohoku Mathematical Journal. Second Series},
    VOLUME = {36},
      YEAR = {1984},
    NUMBER = {4},
     PAGES = {623--646},
      ISSN = {0040-8735,2186-585X},
   MRCLASS = {35Q10 (35D99)},
  MRNUMBER = {767409},
MRREVIEWER = {Michael\ Wiegner},
       DOI = {10.2748/tmj/1178228767},
       URL = {https://doi.org/10.2748/tmj/1178228767},
}

@article {Schonbek1985,
    AUTHOR = {Schonbek, Maria Elena},
     TITLE = {{$L^2$} decay for weak solutions of the {N}avier-{S}tokes
              equations},
   JOURNAL = {Arch. Rational Mech. Anal.},
  FJOURNAL = {Archive for Rational Mechanics and Analysis},
    VOLUME = {88},
      YEAR = {1985},
    NUMBER = {3},
     PAGES = {209--222},
      ISSN = {0003-9527},
   MRCLASS = {35Q10 (76D05)},
  MRNUMBER = {775190},
MRREVIEWER = {Yoshikazu\ Giga},
       DOI = {10.1007/BF00752111},
       URL = {https://doi.org/10.1007/BF00752111},
}

@article {Schonbek1986,
    AUTHOR = {Schonbek, Maria E.},
     TITLE = {Large time behaviour of solutions to the {N}avier-{S}tokes
              equations},
   JOURNAL = {Comm. Partial Differential Equations},
  FJOURNAL = {Communications in Partial Differential Equations},
    VOLUME = {11},
      YEAR = {1986},
    NUMBER = {7},
     PAGES = {733--763},
      ISSN = {0360-5302,1532-4133},
   MRCLASS = {35Q10 (76D05)},
  MRNUMBER = {837929},
MRREVIEWER = {Charles\ J.\ Amick},
       DOI = {10.1080/03605308608820443},
       URL = {https://doi.org/10.1080/03605308608820443},
}

@article {SSS1996,
    AUTHOR = {Schonbek, M. E. and Schonbek, T. P. and S\"uli, Endre},
     TITLE = {Large-time behaviour of solutions to the magnetohydrodynamics
              equations},
   JOURNAL = {Math. Ann.},
  FJOURNAL = {Mathematische Annalen},
    VOLUME = {304},
      YEAR = {1996},
    NUMBER = {4},
     PAGES = {717--756},
      ISSN = {0025-5831,1432-1807},
   MRCLASS = {35Q35 (76W05)},
  MRNUMBER = {1380452},
MRREVIEWER = {Jos\'e\ Luiz\ Boldrini},
       DOI = {10.1007/BF01446316},
       URL = {https://doi.org/10.1007/BF01446316},
}

@article {AgapitoSchonbek2007,
    AUTHOR = {Agapito, Rub\'en and Schonbek, Maria},
     TITLE = {Non-uniform decay of {MHD} equations with and without magnetic
              diffusion},
   JOURNAL = {Comm. Partial Differential Equations},
  FJOURNAL = {Communications in Partial Differential Equations},
    VOLUME = {32},
      YEAR = {2007},
    NUMBER = {10-12},
     PAGES = {1791--1812},
      ISSN = {0360-5302,1532-4133},
   MRCLASS = {35Q35 (35B40 76D05)},
  MRNUMBER = {2372488},
MRREVIEWER = {Paolo\ Secchi},
       DOI = {10.1080/03605300701318658},
       URL = {https://doi.org/10.1080/03605300701318658},
}

@article {MiuraTsai2012,
    AUTHOR = {Miura, Hideyuki and Tsai, Tai-Peng},
     TITLE = {Point singularities of 3{D} stationary {N}avier-{S}tokes
              flows},
   JOURNAL = {J. Math. Fluid Mech.},
  FJOURNAL = {Journal of Mathematical Fluid Mechanics},
    VOLUME = {14},
      YEAR = {2012},
    NUMBER = {1},
     PAGES = {33--41},
      ISSN = {1422-6928,1422-6952},
   MRCLASS = {35Q30 (35A20 35D30 76D05)},
  MRNUMBER = {2891188},
MRREVIEWER = {Sergey\ Nikolaevich\ Alekseenko},
       DOI = {10.1007/s00021-010-0046-6},
       URL = {https://doi.org/10.1007/s00021-010-0046-6},
}

@article {KorolevSverak2011,
    AUTHOR = {Korolev, A. and Sver\'ak, V.},
     TITLE = {On the large-distance asymptotics of steady state solutions of
              the {N}avier-{S}tokes equations in 3{D} exterior domains},
   JOURNAL = {Ann. Inst. H. Poincar\'e{} C Anal. Non Lin\'eaire},
  FJOURNAL = {Annales de l'Institut Henri Poincar\'e{} C. Analyse Non
              Lin\'eaire},
    VOLUME = {28},
      YEAR = {2011},
    NUMBER = {2},
     PAGES = {303--313},
      ISSN = {0294-1449,1873-1430},
   MRCLASS = {35Q30 (35B40 76D05)},
  MRNUMBER = {2784073},
MRREVIEWER = {Lorenzo\ Brandolese},
       DOI = {10.1016/j.anihpc.2011.01.003},
       URL = {https://doi.org/10.1016/j.anihpc.2011.01.003},
}

@article {MR952901,
    AUTHOR = {Schmidt, Paul G\"{u}nter},
     TITLE = {On a magnetohydrodynamic problem of {E}uler type},
   JOURNAL = {J. Differential Equations},
  FJOURNAL = {Journal of Differential Equations},
    VOLUME = {74},
      YEAR = {1988},
    NUMBER = {2},
     PAGES = {318--335},
      ISSN = {0022-0396,1090-2732},
   MRCLASS = {76W05 (35Q10)},
  MRNUMBER = {952901},
MRREVIEWER = {Denis\ Serre},
       DOI = {10.1016/0022-0396(88)90008-3},
       URL = {https://doi.org/10.1016/0022-0396(88)90008-3},
}

@article {MR1257135,
    AUTHOR = {Secchi, Paolo},
     TITLE = {On the equations of ideal incompressible magnetohydrodynamics},
   JOURNAL = {Rend. Sem. Mat. Univ. Padova},
  FJOURNAL = {Rendiconti del Seminario Matematico della Universit\`a di
              Padova. The Mathematical Journal of the University of Padova},
    VOLUME = {90},
      YEAR = {1993},
     PAGES = {103--119},
      ISSN = {0041-8994},
   MRCLASS = {76W05 (35Q30 76C99)},
  MRNUMBER = {1257135},
MRREVIEWER = {Alberto\ Valli},
       URL = {http://www.numdam.org/item?id=RSMUP_1993__90__103_0},
}

@article {MR3415680,
    AUTHOR = {Chemin, Jean-Yves and McCormick, David S. and Robinson, James
              C. and Rodrigo, Jose L.},
     TITLE = {Local existence for the non-resistive {MHD} equations in
              {B}esov spaces},
   JOURNAL = {Adv. Math.},
  FJOURNAL = {Advances in Mathematics},
    VOLUME = {286},
      YEAR = {2016},
     PAGES = {1--31},
      ISSN = {0001-8708,1090-2082},
   MRCLASS = {35Q35 (35A01 35A02 35D30 76W05)},
  MRNUMBER = {3415680},
MRREVIEWER = {Reinhard\ Redlinger},
       DOI = {10.1016/j.aim.2015.09.004},
       URL = {https://doi.org/10.1016/j.aim.2015.09.004},
}

@article {MR3590662,
    AUTHOR = {Fefferman, Charles L. and McCormick, David S. and Robinson,
              James C. and Rodrigo, Jose L.},
     TITLE = {Local existence for the non-resistive {MHD} equations in
              nearly optimal {S}obolev spaces},
   JOURNAL = {Arch. Ration. Mech. Anal.},
  FJOURNAL = {Archive for Rational Mechanics and Analysis},
    VOLUME = {223},
      YEAR = {2017},
    NUMBER = {2},
     PAGES = {677--691},
      ISSN = {0003-9527,1432-0673},
   MRCLASS = {35Q35 (76W05)},
  MRNUMBER = {3590662},
MRREVIEWER = {Alberto\ Valli},
       DOI = {10.1007/s00205-016-1042-7},
       URL = {https://doi.org/10.1007/s00205-016-1042-7},
}

@article {MR2772455,
    AUTHOR = {Zhou, Yong and Fan, Jishan},
     TITLE = {A regularity criterion for the 2{D} {MHD} system with zero
              magnetic diffusivity},
   JOURNAL = {J. Math. Anal. Appl.},
  FJOURNAL = {Journal of Mathematical Analysis and Applications},
    VOLUME = {378},
      YEAR = {2011},
    NUMBER = {1},
     PAGES = {169--172},
      ISSN = {0022-247X,1096-0813},
   MRCLASS = {35Q35 (35B60 35B65 42B35 76D03 76W05)},
  MRNUMBER = {2772455},
MRREVIEWER = {Michael\ Mudi\ Tom},
       DOI = {10.1016/j.jmaa.2011.01.014},
       URL = {https://doi.org/10.1016/j.jmaa.2011.01.014},
}

@article {MR3217057,
    AUTHOR = {Fefferman, Charles L. and McCormick, David S. and Robinson,
              James C. and Rodrigo, Jose L.},
     TITLE = {Higher order commutator estimates and local existence for the
              non-resistive {MHD} equations and related models},
   JOURNAL = {J. Funct. Anal.},
  FJOURNAL = {Journal of Functional Analysis},
    VOLUME = {267},
      YEAR = {2014},
    NUMBER = {4},
     PAGES = {1035--1056},
      ISSN = {0022-1236,1096-0783},
   MRCLASS = {35Q35 (76W05)},
  MRNUMBER = {3217057},
MRREVIEWER = {Paolo\ Secchi},
       DOI = {10.1016/j.jfa.2014.03.021},
       URL = {https://doi.org/10.1016/j.jfa.2014.03.021},
}

@article {MR2265204,
    AUTHOR = {Jiu, Quansen and Niu, Dongjuan},
     TITLE = {Mathematical results related to a two-dimensional
              magneto-hydrodynamic equations},
   JOURNAL = {Acta Math. Sci. Ser. B (Engl. Ed.)},
  FJOURNAL = {Acta Mathematica Scientia. Series B. English Edition},
    VOLUME = {26},
      YEAR = {2006},
    NUMBER = {4},
     PAGES = {744--756},
      ISSN = {0252-9602,1572-9087},
   MRCLASS = {35Q35 (35B40 35Q60 76D03 76W05)},
  MRNUMBER = {2265204},
       DOI = {10.1016/S0252-9602(06)60101-X},
       URL = {https://doi.org/10.1016/S0252-9602(06)60101-X},
}

@article {MR4496608,
    AUTHOR = {Ye, Wei Kui and Yin, Zhao Yang},
     TITLE = {Global well-posedness for the non-viscous {MHD} equations with
              magnetic diffusion in critical {B}esov spaces},
   JOURNAL = {Acta Math. Sin. (Engl. Ser.)},
  FJOURNAL = {Acta Mathematica Sinica (English Series)},
    VOLUME = {38},
      YEAR = {2022},
    NUMBER = {9},
     PAGES = {1493--1511},
      ISSN = {1439-8516,1439-7617},
   MRCLASS = {35Q35 (76W05)},
  MRNUMBER = {4496608},
       DOI = {10.1007/s10114-022-1400-3},
       URL = {https://doi.org/10.1007/s10114-022-1400-3},
}

@article {MR2507450,
    AUTHOR = {Fan, Jishan and Ozawa, Tohru},
     TITLE = {Regularity criteria for the magnetohydrodynamic equations with
              partial viscous terms and the {L}eray-{$\alpha$}-{MHD} model},
   JOURNAL = {Kinet. Relat. Models},
  FJOURNAL = {Kinetic and Related Models},
    VOLUME = {2},
      YEAR = {2009},
    NUMBER = {2},
     PAGES = {293--305},
      ISSN = {1937-5093,1937-5077},
   MRCLASS = {76D03 (35B65)},
  MRNUMBER = {2507450},
       DOI = {10.3934/krm.2009.2.293},
       URL = {https://doi.org/10.3934/krm.2009.2.293},
}

@article {MR1007099,
    AUTHOR = {Kozono, Hideo},
     TITLE = {Weak and classical solutions of the two-dimensional
              magnetohydrodynamic equations},
   JOURNAL = {Tohoku Math. J. (2)},
  FJOURNAL = {The Tohoku Mathematical Journal. Second Series},
    VOLUME = {41},
      YEAR = {1989},
    NUMBER = {3},
     PAGES = {471--488},
      ISSN = {0040-8735,2186-585X},
   MRCLASS = {35Q10 (76W05)},
  MRNUMBER = {1007099},
MRREVIEWER = {Sadakazu\ Aizawa},
       DOI = {10.2748/tmj/1178227774},
       URL = {https://doi.org/10.2748/tmj/1178227774},
}

@article{zbMATH07327027,
 author = {Beekie, Rajendra and Buckmaster, Tristan and Vicol, Vlad},
 title = {Weak solutions of ideal {MHD} which do not conserve magnetic helicity},
 fjournal = {Annals of PDE},
 journal = {Ann. PDE},
 issn = {2524-5317},
 volume = {6},
 number = {1},
 pages = {40},
 note = {Id/No 1},
 year = {2020},
 doi = {10.1007/s40818-020-0076-1},
 zbMATH = {7327027},
 Zbl = {1462.35282}
}

@article{zbMATH08070132,
 author = {Bradshaw, Zachary and Wang, Weinan},
 title = {Asymptotic stability for the 3D {Navier}-{Stokes} equations in {{\(L^3\)}} and nearby spaces},
 fjournal = {Proceedings of the American Mathematical Society},
 journal = {Proc. Am. Math. Soc.},
 issn = {0002-9939},
 volume = {153},
 number = {9},
 pages = {3867--3881},
 year = {2025},
 doi = {10.1090/proc/17301},
 zbMATH = {8070132},
 Zbl = {1569.35199}
}

@article{LiYan2021,
 author = {Li, YanYan and Yan, Xukai},
 title = {Asymptotic stability of homogeneous solutions of incompressible stationary {Navier}-{Stokes} equations},
 fjournal = {Journal of Differential Equations},
 journal = {J. Differ. Equations},
 issn = {0022-0396},
 volume = {297},
 pages = {226--245},
 year = {2021},
 doi = {10.1016/j.jde.2021.06.033},
 zbMATH = {7373375},
 Zbl = {1473.35407}
}

@article{zbMATH07471753,
 author = {Cannone, Marco and Karch, Grzegorz and Pilarczyk, Dominika and Wu, Gang},
 title = {Stability of singular solutions to the {Navier}-{Stokes} system},
 fjournal = {Journal of Differential Equations},
 journal = {J. Differ. Equations},
 issn = {0022-0396},
 volume = {314},
 pages = {316--339},
 year = {2022},
 doi = {10.1016/j.jde.2022.01.010},
 zbMATH = {7471753},
 Zbl = {1487.35283}
}

@article{LZZ23,
 author = {Li, Yanyan and Zhang, Jingjing and Zhang, Ting},
 title = {Asymptotic stability of {Landau} solutions to {Navier}-{Stokes} system under {{\(L^p\)}}-perturbations},
 fjournal = {Journal of Mathematical Fluid Mechanics},
 journal = {J. Math. Fluid Mech.},
 issn = {1422-6928},
 volume = {25},
 number = {1},
 pages = {30},
 note = {Id/No 5},
 year = {2023},
 doi = {10.1007/s00021-022-00751-x},
 zbMATH = {7638115},
 Zbl = {1532.76036}
}

@article {MR1777114,
    AUTHOR = {Yamazaki, Masao},
     TITLE = {The {N}avier-{S}tokes equations in the weak-{$L^n$} space with
              time-dependent external force},
   JOURNAL = {Math. Ann.},
  FJOURNAL = {Mathematische Annalen},
    VOLUME = {317},
      YEAR = {2000},
    NUMBER = {4},
     PAGES = {635--675},
      ISSN = {0025-5831,1432-1807},
   MRCLASS = {35Q30 (35A07 76D03 76D05)},
  MRNUMBER = {1777114},
MRREVIEWER = {J\"{u}rgen\ Socolowsky},
       DOI = {10.1007/PL00004418},
       URL = {https://doi.org/10.1007/PL00004418},
}

@article{zbMATH06383463,
 author = {Hmidi, Taoufik},
 title = {On the {Yudovich} solutions for the ideal {MHD} equations},
 fjournal = {Nonlinearity},
 journal = {Nonlinearity},
 issn = {0951-7715},
 volume = {27},
 number = {12},
 pages = {3117--3158},
 year = {2014},
 doi = {10.1088/0951-7715/27/12/3117},
 zbMATH = {6383463},
 Zbl = {1374.35310}
}

@article{zbMATH05675031,
 author = {Chen, Qionglei and Miao, Changxing and Zhang, Zhifei},
 title = {On the well-posedness of the ideal {MHD} equations in the {Triebel}-{Lizorkin} spaces},
 fjournal = {Archive for Rational Mechanics and Analysis},
 journal = {Arch. Ration. Mech. Anal.},
 issn = {0003-9527},
 volume = {195},
 number = {2},
 pages = {561--578},
 year = {2010},
 doi = {10.1007/s00205-008-0213-6},
 zbMATH = {5675031},
 Zbl = {1184.35258}
}

@article{zbMATH07635066,
 author = {Cobb, Dimitri and Fanelli, Francesco},
 title = {Els{\"a}sser formulation of the ideal {MHD} and improved lifespan in two space dimensions},
 fjournal = {Journal de Math{\'e}matiques Pures et Appliqu{\'e}es. Neuvi{\`e}me S{\'e}rie},
 journal = {J. Math. Pures Appl. (9)},
 issn = {0021-7824},
 volume = {169},
 pages = {189--236},
 year = {2023},
 doi = {10.1016/j.matpur.2022.11.012},
 zbMATH = {7635066},
 Zbl = {1504.35304}
}

@article{zbMATH08053923,
 author = {Liu, Sicheng and Xin, Zhouping},
 title = {Local well-posedness of the incompressible current-vortex sheet problems},
 fjournal = {Advances in Mathematics},
 journal = {Adv. Math.},
 issn = {0001-8708},
 volume = {475},
 pages = {90},
 note = {Id/No 110339},
 year = {2025},
 doi = {10.1016/j.aim.2025.110339},
 zbMATH = {8053923},
 Zbl = {1569.35217}
}

@article {MR894160,
    AUTHOR = {Kozono, Hideo},
     TITLE = {On the energy decay of a weak solution of the {MHD} equations
              in a three-dimensional exterior domain},
   JOURNAL = {Hokkaido Math. J.},
  FJOURNAL = {Hokkaido Mathematical Journal},
    VOLUME = {16},
      YEAR = {1987},
    NUMBER = {2},
     PAGES = {151--166},
      ISSN = {0385-4035},
   MRCLASS = {35Q99 (76W05)},
  MRNUMBER = {894160},
MRREVIEWER = {P.\ L.\ Sulem},
       DOI = {10.14492/hokmj/1381517923},
       URL = {https://doi.org/10.14492/hokmj/1381517923},
}

@article {KarchPilarczykSchonbek17,
    AUTHOR = {Karch, Grzegorz and Pilarczyk, Dominika and Schonbek, Maria
              E.},
     TITLE = {{$L^2$}-asymptotic stability of singular solutions to the
              {N}avier-{S}tokes system of equations in {$\mathbf{R}^3$}},
   JOURNAL = {J. Math. Pures Appl. (9)},
  FJOURNAL = {Journal de Math\'ematiques Pures et Appliqu\'ees. Neuvi\`eme
              S\'erie},
    VOLUME = {108},
      YEAR = {2017},
    NUMBER = {1},
     PAGES = {14--40},
      ISSN = {0021-7824,1776-3371},
   MRCLASS = {35Q30 (35B35 35B40 76D05)},
  MRNUMBER = {3660767},
MRREVIEWER = {Pavel\ I.\ Naumkin},
       DOI = {10.1016/j.matpur.2016.10.008},
       URL = {https://doi.org/10.1016/j.matpur.2016.10.008},
}

@article {LLY18I,
    AUTHOR = {Li, Li and Li, YanYan and Yan, Xukai},
     TITLE = {Homogeneous solutions of stationary {N}avier-{S}tokes
              equations with isolated singularities on the unit sphere. {I}.
              {O}ne singularity},
   JOURNAL = {Arch. Ration. Mech. Anal.},
  FJOURNAL = {Archive for Rational Mechanics and Analysis},
    VOLUME = {227},
      YEAR = {2018},
    NUMBER = {3},
     PAGES = {1091--1163},
      ISSN = {0003-9527,1432-0673},
   MRCLASS = {35Q30 (76D03 76D05)},
  MRNUMBER = {3744383},
MRREVIEWER = {J\"urgen\ Socolowsky},
       DOI = {10.1007/s00205-017-1181-5},
       URL = {https://doi.org/10.1007/s00205-017-1181-5},
}

@article {LLY18II,
    AUTHOR = {Li, Li and Li, YanYan and Yan, Xukai},
     TITLE = {Homogeneous solutions of stationary {N}avier-{S}tokes
              equations with isolated singularities on the unit sphere.
              {II}. {C}lassification of axisymmetric no-swirl solutions},
   JOURNAL = {J. Differential Equations},
  FJOURNAL = {Journal of Differential Equations},
    VOLUME = {264},
      YEAR = {2018},
    NUMBER = {10},
     PAGES = {6082--6108},
      ISSN = {0022-0396,1090-2732},
   MRCLASS = {35Q30 (76D03 76D05)},
  MRNUMBER = {3770045},
MRREVIEWER = {J\"urgen\ Socolowsky},
       DOI = {10.1016/j.jde.2018.01.028},
       URL = {https://doi.org/10.1016/j.jde.2018.01.028},
}

@article {LLY19III,
    AUTHOR = {Li, Li and Li, Yanyan and Yan, Xukai},
     TITLE = {Homogeneous solutions of stationary {N}avier-{S}tokes
              equations with isolated singularities on the unit sphere.
              {III}. {T}wo singularities},
   JOURNAL = {Discrete Contin. Dyn. Syst.},
  FJOURNAL = {Discrete and Continuous Dynamical Systems},
    VOLUME = {39},
      YEAR = {2019},
    NUMBER = {12},
     PAGES = {7163--7211},
      ISSN = {1078-0947,1553-5231},
   MRCLASS = {35Q30 (76D03 76D05)},
  MRNUMBER = {4026186},
MRREVIEWER = {J\"urgen\ Socolowsky},
       DOI = {10.3934/dcds.2019300},
       URL = {https://doi.org/10.3934/dcds.2019300},
}

@incollection {AGG2006,
    AUTHOR = {Arendt, Wolfgang and Goldstein, Gis\`ele Ruiz and Goldstein,
              Jerome A.},
     TITLE = {Outgrowths of {H}ardy's inequality},
 BOOKTITLE = {Recent advances in differential equations and mathematical
              physics},
    SERIES = {Contemp. Math.},
    VOLUME = {412},
     PAGES = {51--68},
 PUBLISHER = {Amer. Math. Soc., Providence, RI},
      YEAR = {2006},
      ISBN = {978-0-8218-3840-2; 0-8218-3840-7},
   MRCLASS = {47D06 (26D15 35J15 47F05 81Q10)},
  MRNUMBER = {2259099},
MRREVIEWER = {Michael\ A.\ Perelmuter},
       DOI = {10.1090/conm/412/07766},
       URL = {https://doi.org/10.1090/conm/412/07766},
}

@book {LR02,
    AUTHOR = {Lemari\'e-Rieusset, P. G.},
     TITLE = {Recent developments in the {N}avier-{S}tokes problem},
    SERIES = {Chapman \& Hall/CRC Research Notes in Mathematics},
    VOLUME = {431},
 PUBLISHER = {Chapman \& Hall/CRC, Boca Raton, FL},
      YEAR = {2002},
     PAGES = {xiv+395},
      ISBN = {1-58488-220-4},
   MRCLASS = {35Q30 (46N20 76D03 76D05)},
  MRNUMBER = {1938147},
MRREVIEWER = {Joel\ David\ Avrin},
       DOI = {10.1201/9781420035674},
       URL = {https://doi.org/10.1201/9781420035674},
}

@article {CKN82,
    AUTHOR = {Caffarelli, L. and Kohn, R. and Nirenberg, L.},
     TITLE = {Partial regularity of suitable weak solutions of the
              {N}avier-{S}tokes equations},
   JOURNAL = {Comm. Pure Appl. Math.},
  FJOURNAL = {Communications on Pure and Applied Mathematics},
    VOLUME = {35},
      YEAR = {1982},
    NUMBER = {6},
     PAGES = {771--831},
      ISSN = {0010-3640,1097-0312},
   MRCLASS = {35Q10 (76D05)},
  MRNUMBER = {673830},
MRREVIEWER = {Tai\ Ping\ Liu},
       DOI = {10.1002/cpa.3160350604},
       URL = {https://doi.org/10.1002/cpa.3160350604},
}

\vfill 

\end{document}